\documentclass[12pt]{amsart}

\usepackage{pdfsync}         
\usepackage[active]{srcltx}  
\usepackage{tikz}
\usepackage{enumerate}
\usepackage{comment}
\usepackage{amssymb}
\usepackage{subcaption}
\usepackage{mathtools}

\usepackage[OT1]{fontenc}    %
\usepackage{graphicx}

\usepackage{palatino}
\usepackage[left=3cm,top=3.8cm,right=3cm]{geometry}        

\usepackage{xcolor}
\usepackage[pagebackref]{hyperref}
\hypersetup{
   colorlinks,
    linkcolor={red!60!black},
    citecolor={blue!60!black},
    urlcolor={blue!90!black}
}

\numberwithin{equation}{section}

\theoremstyle{plain}
\newtheorem{theorem}{Theorem}[section]
\newtheorem{corollary}[theorem]{Corollary}
\newtheorem{prop}[theorem]{Proposition}
\newtheorem{lemma}[theorem]{Lemma}

\theoremstyle{remark}
\newtheorem{remark}[theorem]{Remark}

\theoremstyle{definition}
\newtheorem{definition}[theorem]{Definition}

\newcommand{\leb}{\mathcal{L}}

\newcommand{\e}{\varepsilon}
\newcommand{\N}{\mathbb{N}}
\newcommand{\R}{\mathbb{R}}

\newcommand{\dist}{\mathrm{dist}}
\newcommand{\ol}{\overline}

\newcommand{\cP}{\mathcal{P}}
\newcommand{\cA}{\mathcal{A}}
\newcommand{\cM}{\mathcal{M}}
\newcommand{\cL}{\mathcal{L}}

\newcommand{\cN}{\mathcal{N}}
\newcommand{\cD}{\mathcal{D}}
\newcommand{\cG}{\mathcal{G}}

\newcommand{\cI}{\mathcal{I}}
\newcommand{\cE}{\mathcal{E}}
\newcommand{\cF}{\mathcal{F}}

\newcommand{\cR}{\mathcal{R}}
\newcommand{\cT}{\mathcal{T}}
\newcommand{\cV}{\mathcal{V}}

\DeclareMathOperator{\dir}{dir}

\DeclareMathOperator{\hdim}{dim_H}
\DeclareMathOperator{\pdim}{dim_P}

\DeclareMathOperator{\supp}{supp}

\DeclareMathOperator{\bad}{\mathbf{Bad}}

\newcommand{\wh}{\widehat}
\newcommand{\wt}{\widetilde}

\title{A nonlinear version of Bourgain's projection theorem}

\author{Pablo Shmerkin}

\address{Department of Mathematics, the University of British Columbia, Canada}
\email{pshmerkin@math.ubc.ca}
\urladdr{http://pabloshmerkin.org}
\thanks{PS has received funding from a University of St Andrews Global Fellowship and from the European Research Council (ERC) under the European Union's Horizon 2020 research and innovation programme (grant agreement No. 803711)}


\begin{document}

\begin{abstract}
We prove a version of Bourgain's projection theorem for parametrized families of $C^2$ maps, that refines the original statement even in the linear case by requiring non-concentration only at a single natural scale. As one application, we show that if $A$ is a Borel set of Hausdorff dimension close to $1$ in $\mathbb{R}^2$ or close to $3/2$ in $\mathbb{R}^3$, then for $y\in A$ outside of a very sparse set,  the pinned distance set $\{|x-y|:x\in A\}$ has Hausdorff dimension at least $1/2+c$, where $c$ is universal. Furthermore, the same holds if the distances are taken with respect to a $C^2$ norm of positive Gaussian curvature. As further applications, we obtain new bounds on the dimensions of spherical projections, and an improvement over the trivial estimate for incidences between $\delta$-balls and $\delta$-neighborhoods of curves in the plane, under fairly general assumptions. The proofs depend on a new multiscale decomposition of measures into ``Frostman pieces'' that may be of independent interest.
\end{abstract}

\dedicatory{Dedicated to the memory of Jean Bourgain}

\maketitle

\tableofcontents

\section{Introduction and statement of results}

\subsection{Distance sets}

The Falconer distance set conjecture, originating in \cite{Falconer85}, asserts that if $A\subset\R^d$, $d\ge 2$, is a Borel set of Hausdorff dimension $d/2$, then the \emph{distance set}
\[
\Delta(A) =\{|x-y|:x,y\in A\} \subset [0,\infty)
\]
has Hausdorff dimension $1$. Despite efforts by many mathematicians, the conjecture remains open in all dimensions.  A variant of this problem has also received much attention. From now on, all sets are assumed to be Borel.  Given $y\in\R^d, A\subset\R^d$, denote the \emph{pinned distance set} by
\[
\Delta^y(A)=\{ |x-y|:x\in A\}.
\]
Let $\hdim$ denote Hausdorff dimension.  It is possible that the pinned version of Falconer's conjecture holds, that is, if $\hdim(A)\ge d/2$, then there is $y\in A$ such that $\hdim(\Delta^y(A))=1$.

Recent deep results \cite{GIOW20, DuZhang19, Liu19} imply that if $\hdim(A)>\alpha(d)$, then there is $y\in A$ such that $\Delta^y(A)$ has positive Lebesgue measure, where
\begin{equation} \label{eq:alpha-d}
\alpha(d) = \left\{ \begin{array}{ccc}
                 \frac{5}{4} & \text{ if } & d=2 \\
                 \frac{d^2}{2d-1} & \text{ if } & d\ge 3
               \end{array}
 \right..
\end{equation}
In recent years, substantial progress has been achieved in the plane under the assumption $\hdim(A)>1$ \cite{Orponen17,Shmerkin17,Shmerkin21, KeletiShmerkin19,Shmerkin19, GIOW20, Liu20}. For example, it is known \cite{Shmerkin19} that if $A$ is a planar set of equal Hausdorff and packing dimension, and this common value is $>1$, then there is $y\in A$ such that $\hdim(\Delta^y(A))=1$. For general planar sets of Hausdorff dimension $>1$, it is known that there is $y\in A$ such that $\hdim(\Delta^y(A))\ge 29/42$ \cite{Shmerkin21} and $\pdim(\Delta^y(A))> 0.933$ \cite{KeletiShmerkin19}, where $\pdim$ denotes packing dimension. In all these works it is crucial that $\hdim(A)>1$; the methods break down if one only assumes that $\hdim(A)=1$.

For general ambient dimension, under the hypothesis $\hdim(A)=d/2$, Falconer proved in his original paper \cite{Falconer85} that $\hdim(\Delta(A))\ge 1/2$. This bound turned out to be quite hard to improve upon. By combining results of N.~Katz-T.~Tao \cite{KatzTao01} and J.~Bourgain \cite{Bourgain03}, it follows that if $A$ is a Borel planar set with $\hdim(A)\ge 1$, then
\begin{equation} \label{eq:katz-tao-bourgain}
\hdim(\Delta(A))\ge 1/2 +c,
\end{equation}
where $c$ is a small universal constant. We discuss the result of Katz-Tao and Bourgain in more detail below.

Less attention has been given to the problem of obtaining lower bounds for the dimension of distance sets when $\hdim(A)<d/2$. To our knowledge, the best such bound appearing in the literature was proved in Falconer's paper \cite{Falconer85}: if $\hdim(A)\le d/2$, then
\begin{equation} \label{eq:falconer-dim-less-d-over-2}
\hdim(\Delta(A)) \ge \hdim(A) - \frac{d-1}{2}.
\end{equation}
Note that this is vacuous if $\hdim(A)\le (d-1)/2$.

In this paper we improve upon several of the previously mentioned results, particularly in dimensions $2$ and $3$.
\begin{theorem} \label{thm:main-distance-sets}
Let $d\ge 2$. Given $d-2<\kappa,\alpha<d$, there is a number $c=c_d(\alpha,\kappa)>0$, depending continuously on $\alpha,\kappa$, such that the following holds.

Let $A\subset\R^d$ be a Borel set with $\hdim(A)\ge \alpha$. Then
\[
\hdim\left( \left\{y\in\R^d : \hdim(\Delta^y(A))< \frac{\alpha}{d}+c\right\}\right) \le \kappa.
\]
Moreover, the same holds if the pinned distance set is defined with respect to any $C^2$ norm whose unit ball has everywhere positive Gaussian curvature (with $c$ independent of the choice of norm).
\end{theorem}

In the plane, this result improves upon the bound \eqref{eq:katz-tao-bourgain} of Katz-Tao and Bourgain in several ways: (a) it provides a pinned version, (b) furthermore, not only does the pinned version hold for some $y\in A$, but in fact it holds for $y$ outside of a set of dimension $\kappa$, arbitrarily small (with the gain $c$ depending on $\kappa$) - this is new even when $\hdim(A)>1$, (c) it works for more general smooth, curved norms, (d) it extends to values of $\hdim(A)\in (0,1)$. Moreover, Theorem \ref{thm:main-distance-sets} provides an improvement over the bound $\hdim(\Delta(A))\ge 1/2$ for Borel sets of dimension $3/2$ in $\R^3$, and also for Borel sets of dimension $>2$ in $\R^4$. In dimensions $\ge 5$, the value $\alpha(d)$ from \eqref{eq:alpha-d} is smaller than $d-2$, so Theorem \ref{thm:main-distance-sets} becomes far less interesting. We note that a careful inspection of the proofs in \cite{KatzTao01}, combined with the discretized sum-product theorem of \cite{Bourgain03}, might lead to the analogue of \eqref{eq:katz-tao-bourgain} also for sets of dimension in $(0,1)$ and perhaps also in higher dimensions. We believe that the improvements (a)--(c) noted above do not follow in any direct way from previous approaches.

\begin{remark}
In dimensions $d\ge 3$, at least one of the assumptions $\kappa>d-2$, $\alpha>d-2$ in Theorem \ref{thm:main-distance-sets} is necessary, as can be seen from the examples $X=S^1\times \{0\}\subset \R^2\times\R^{d-2}$ and $X=S^{d-2}\times \{0\}\subset\R^{d-1}\times\R$. However, such assumptions may not be necessary if one only considers $y\in X$.
\end{remark}

\subsection{A non-linear version of Bourgain's projection theorem}

There are well known connections between many important problems at the interface of analysis and geometric measure theory, such as the Kakeya, Furstenberg set, discretized sum-product, discretized projection and Falconer distance set problems. Indeed, all of these problems to some extent deal with incidences between tubes. However, the connections, even when they are explicit, are rarely straightforward. In \cite{KatzTao01}, Katz and Tao introduced discretized versions of three conjectures which were at the time open. Two of the conjectures were: (a) $\hdim(\Delta(A))\ge 1/2+c$ if $A$ is a planar Borel set with $\hdim(A)=1$, (b) there is no Borel subring of the reals of Hausdorff dimension $1/2$. Note that (a)$\Rightarrow$(b): if $R$ is a ring of dimension $1/2$, then $\Delta(R\times R)\subset \sqrt{R_{\ge 0}}$. Among other things, Katz and Tao proved that certain discretized versions of these conjectures are equivalent, and that the discretized version of \eqref{eq:katz-tao-bourgain} (which is rather involved) implies the actual bound \eqref{eq:katz-tao-bourgain}.

In \cite{Bourgain03}, Bourgain proved the discretized version of the ring conjecture, which is nowadays known as the discretized sum-product theorem. Hence, in combination with \cite{KatzTao01}, this established the bound \eqref{eq:katz-tao-bourgain}. A few years later, in \cite{Bourgain10}, Bourgain refined the discretized sum-product theorem to obtain what is now known as Bourgain's (discretized) projection theorem. We recall this theorem below. (In fact, many of the ideas to go from sum-product to projections are already implicit in \cite{KatzTao01, Bourgain03}.) Thus, there is a known path from Bourgain's projection theorem to the estimate \eqref{eq:katz-tao-bourgain}. In this article, we take a rather different path that we believe is more flexible, and can be used to make progress on other problems in combinatorial fractal geometry. We view the maps $\Delta^y(x)=|x-y|$ as a family of maps from $\R^d\to \R$, parametrized by the point $y$. These maps are smooth if we are careful to separate the domains of the $x$ and the $y$, but they are nonlinear. In \cite{KatzTao01}, an important step in the overall argument is applying a projective transformation that linearizes a family of projections. This argument seems to be rather constrained. The approach of this paper is to develop a \emph{non-linear} version of Bourgain's projection theorem. In doing so, we also obtain some new insights even in the linear case.

We state first a continuous (as opposed to single-scale) version of our main result. We denote the open $r$-neighborhood of $H$ by $H^{(r)}$. If $\mu$ is a Borel measure on a metric space $X$ and $g:X\to Y$ is a Borel map, then we denote the push-forward measure to $Y$ by $g\mu$, that is, $g\mu(\cdot)=\mu(g^{-1}\cdot)$.  If $\mu(A)>0$, then $\mu_A$ denotes the normalized restriction $\mu(A)^{-1}\mu|_A$.
Given $x\in \R^d, x\neq 0$, we let $\dir(x)=x/|x|\in S^{d-1}$. Finally, we denote the Grassmanian of linear $k$-planes in $\R^d$ by $\mathbb{G}(d,k)$.
\begin{theorem} \label{thm:main-continuous}
Given $\kappa>0, 0<\alpha<d$ there is $\eta=\eta_d(\kappa,\alpha)>0$ (that can be taken continuous in $\kappa,\alpha$) such that the following holds.

Let $(\Lambda,\nu)$ be a Borel probability space and let $U\subset\R^d$ be an open domain. Let $F:\Lambda \times U\to\R$ be a Borel function such that, for each $\lambda\in\Lambda$, the map $F_\lambda(x)=F(\lambda,x):U\to\R$ is $C^2$ without singular points. For each $x\in \R^d$  define the map
\[
\theta_x(\lambda) = \dir(\nabla F_\lambda(x)) \in S^{d-1}.
\]
Let $A\subset U$ be a Borel set of dimension $\ge \alpha$ such that for all $x\in A$ there are a set $\Lambda_x$ with $\nu(\Lambda_x)>0$ and a number $C_x>0$ satisfying
\begin{equation} \label{eq:main-thm-cont-projective-decay}
\theta_x\nu_{\Lambda_x}\left(H^{(r)}\right)   \le C_x \, r^\kappa \quad\text{for all }H\in \mathbb{G}(d,d-1), r\in (0,1].
\end{equation}
Then there is $\lambda\in\supp(\nu)$ such that
\[
\hdim(F_\lambda A) \ge \frac{\alpha}{d}+\eta.
\]
\end{theorem}

We make some remarks on this statement.
\begin{remark}
Bourgain's (continuous) projection theorem corresponds precisely to the special case in which $\Lambda=\Lambda_x=S^{d-1}$ and $F_\lambda(x)=\langle \lambda, x\rangle$ is orthogonal projection in direction $\lambda$. In this case, $\theta_x$ is the identity map for all $x$, and so the decay condition \eqref{eq:main-thm-cont-projective-decay} has to be satisfied by the measure $\nu$ on $S^{d-1}$.
\end{remark}

\begin{remark}
Allowing the set $\Lambda_x$ to depend on $x$ is important in our applications, such as Theorem \ref{thm:main-distance-sets}. By a formal argument, in the case $\Lambda_x\equiv \Lambda$ for all $x$, the conclusion holds for $\nu$-almost all $\lambda$.
\end{remark}

\begin{remark}
It is enough that \eqref{eq:main-thm-cont-projective-decay} holds for all $x\in A$ outside of a set $E$ of dimension $<\alpha$, since we can then apply the theorem to $A\setminus E$.
\end{remark}

Bourgain's projection theorem described above is deduced from a discretized version, that is fully stated as Theorem \ref{thm:projection-Bourgain} below. It is often this discretized version that gets used in applications, such as \cite{BFLM11, KatzZahl19}. Correspondingly, we have the following discretized version of Theorem \ref{thm:main-continuous}. The number of dyadic cubes of side length $2^{-m}$ hitting a set $X$ is denoted by $\cN(X,m)$, and $|\cdot|$ refers to Lebesgue measure.

\begin{theorem} \label{thm:main-discrete}
Fix $d\ge 2$. Given $\kappa>0$, there is $\eta=\eta_d(\kappa)>0$ such that the following holds for $\e<\e_d(\kappa)$.

Let $X\subset [0,1]^d$ be a union of $2^{-m}$-dyadic cubes.  Let $U$ be a neighborhood of $X$ and let $(\Lambda,\nu)$ be a Borel probability space. Let $F:\Lambda \times U\to\R$ be a Borel function such that, for each $\lambda\in\Lambda$, the map $F_\lambda(x)=F(\lambda,x):U\to\R$ is $C^2$ without singular points. For each $x\in \R^d$  define the map
\[
\theta_x(\lambda) = \dir(\nabla F_\lambda(x)) \in S^{d-1}.
\]
We assume that $C_\Lambda=\sup_\lambda\|F_\lambda\|_{C^2}<\infty$ and $c_\Lambda=\inf_\lambda \inf_{x\in X}|\nabla F_\lambda(x)|>0$.

Let $m$ be large enough in terms of $c_\Lambda, C_\Lambda$ and all the previous parameters. Suppose $X$ satisfies the single-scale non-concentration condition
\begin{equation} \label{eq:hyp-non-concentration}
|X\cap B(x,|X|^{1/d})| \le 2^{-\kappa m}|X| \quad\text{for all }x.
\end{equation}
For each $x\in X$  define the map
\[
\theta_x(\lambda) = \dir(\nabla F_\lambda(x)) \in S^{d-1}.
\]
Suppose that for every $x\in X$ there is a set $\Lambda_x\subset\Lambda$ with $\nu(\Lambda_x)\ge 2^{-\e m}$ such that
\begin{equation} \label{eq:hyp-circle-non-concentration}
\theta_x\nu_{\Lambda_x}\left(H^{(r)}\right) \le 2^{\e m}r^{\kappa}
\end{equation}
for all hyperplanes $H\in \mathbb{G}(d,d-1)$ and all $r\in [2^{-m},1]$.

Then there exist $\lambda\in \supp(\nu)$ and a set $X'\subset X$ with $|X'|\ge 2^{-2\e m}|X|$ such that
\[
\cN(F_\lambda X'',m) \ge 2^{\eta m} \cN(X,m)^{1/d}.
\]
for all sets $X''\subset X'$ with $|X''|\ge 2^{-\e m}|X'|$.

In the case in which $\Lambda_x=\Lambda$ for all $x$, one can take $X'=X$, and the conclusion holds for all $\lambda$ outside of a set of $\nu$-measure $\le 2^{-\e m}$.
\end{theorem}

Some remarks are in order.
\begin{remark}
Again, Bourgain's original theorem corresponds to $F_\lambda(x)=\langle \lambda,x\rangle$ and $\Lambda_x\equiv\Lambda$, see Theorem \ref{thm:projection-Bourgain} below. However, hypothesis \eqref{eq:hyp-non-concentration} is weaker than the non-concentration assumption in Theorem \ref{thm:projection-Bourgain}, since it is required at the single scale $|X|^{1/d}$. Furthermore, this is the natural scale that makes this assumption sharp: if $X$ is a cube of side length $|X|^{1/d}$ (or dense in such a cube, or the union of very few such cubes), then $|F(X)|\approx |X|^{1/d}$ for all smooth maps $F:\R^d\to\R$ without singular points (including projections). In short, \eqref{eq:hyp-non-concentration} is saying that $X$ is not concentrated in a few cubes, with the ``measure gain'' exponent $\eta$ depending on the quality of this non-concentration, given by $\kappa$. We remark that the theorem clearly fails if $X$ has measure close to either $1$ or to $2^{-dm}$; however, in both these cases $X$ is highly concentrated in a cube (of size $1$ or $2^{-m}$) and therefore these extremes are ruled out by \eqref{eq:hyp-non-concentration}.
\end{remark}

\begin{remark}
In the case $d=2$, the non-concentration assumption \eqref{eq:hyp-circle-non-concentration} is a standard Frostman (power law) assumption on the $\theta_x\nu_{\Lambda_x}$ measures of balls. For $d\ge 3$, the intersection $S^{d-1}\cap H$ is a maximal $(d-2)$-sphere in $S^{d-1}$, and \eqref{eq:hyp-circle-non-concentration} says that $\theta_x\nu_{\Lambda_x}$ is not concentrated near such sub-spheres. The example of a segment shows that, already for linear projections, such a condition is necessary.
\end{remark}

\begin{remark}
The ``gain''  $\eta$ in Bourgain's theorem is effective in principle, but extremely small. The value of $\eta$ in Theorem \ref{thm:main-discrete} is even smaller than that of Bourgain's projection theorem (as a matter of fact, it is equal for sets with some uniform decay such as Ahlfors-regular sets). We do not make the connection explicit, although it is not hard to extract it from the proofs, and we make no attempt at optimization since the values involved are in any event tiny. In the work in progress \cite{HLV20}, the authors provide explicit estimates for a closely related result involving entropy gain. Even though this does not automatically translate into an explicit value in the context of Theorem \ref{thm:projection-Bourgain}, it is plausible that with some additional work this will yield explicit estimates in Theorem \ref{thm:main-discrete}.
\end{remark}

\subsection{Strategy of proof and a multiscale decomposition intro Frostman pieces}

We now explain the overall strategy of the proof of Theorems \ref{thm:main-continuous} and \ref{thm:main-discrete}, and discuss informally a new multiscale decomposition of measures that is at the heart of the proof (see Section \ref{sec:multiscale-decompositions} and in particular Theorem \ref{thm:multiscale-decomposition} for precise statements).

Bourgain's original proof of Theorem \ref{thm:projection-Bourgain} appears to be intrinsically restricted to the linear setting. Rather than modifying the proof, we perform a regularization and multiscale decomposition of a Frostman measure $\mu$ on the given set $X$, and then linearize $F$ at every scale in the multiscale decomposition. We then apply Bourgain's Theorem (as a black box) to every small piece of $\mu$ and every scale in the multiscale decomposition.

The overall strategy is not new. In particular, we make use of a variant of a formula for estimating the entropy of smooth projections in terms of multiscale decompositions that goes back, in various forms, to \cite{HochmanShmerkin12, Hochman14, Orponen17, KeletiShmerkin19}. See Proposition \ref{prop:good-bound-from-below}. Because the precise formulation we need does not appear in the literature, we include a full proof in Appendix \ref{sec:entropy-of-projections}. Our approach is closest to that of \cite{KeletiShmerkin19} which, in addition to a multiscale decomposition, required an initial decomposition of the measure $\mu$ into ``regular Moran constructions''. The same decomposition plays a key role in this paper.

In order to apply Bourgain's projection theorem to small pieces of the measure, we need to verify that these small pieces satisfy the required non-concentration assumption. The main technical innovation of this paper is a new multiscale decomposition of a ``regular Moran measure'' so that the conditional measures on small cubes of certain sizes are Frostman measures up to a small exponential error (so they satisfy essentially the strongest possible non-concentration decay). By itself, this is not enough to deduce any useful consequences, because it may well happen that nearly all of these conditional measures look like either Lebesgue measure or an atom so that, even though they are trivially Frostman measures, there is no gain to be achieved. To deal with this, we show that, if the original measure satisfies the (minimal) non-concentration assumption \eqref{eq:hyp-non-concentration}, then the scales in the multi-scale decomposition can be chosen so that, for a ``positive density'' set of scales, the conditional measures have ``intermediate size'', i.e. they are quantitatively separated from both Lebesgue and atomic measures. See Theorem \ref{thm:multiscale-decomposition} for details. We hope this new multiscale decomposition of measures (and related ones) will have further applications.

Even with Theorem \ref{thm:multiscale-decomposition} in hand, it is still possible that for many scales the conditional measures do look close to Lebesgue or atomic. In this setting, Bourgain's projection theorem does not apply. For such scales, we rely on quantitative versions of more classical projection theorems, essentially going back to R.~Kaufman \cite{Kaufman68} and to K.~Falconer \cite{Falconer82}. See Section \ref{sec:projection-theorems}.

Incidentally, as the above sketch indicates, we only need to apply Theorem \ref{thm:projection-Bourgain} when $\kappa$ (nearly) matches the size of the set being projected, i.e. when there is (near) optimal non-concentration. Potentially it may be easier to obtain quantitative estimates under this stronger assumption, which would then translate into quantitative estimates under the weaker assumptions of Theorem \ref{thm:main-discrete}.

To deduce Theorem \ref{thm:main-distance-sets} from Theorem \ref{thm:main-continuous}, we appeal to a result of T.~Orponen \cite{Orponen19} on spherical projections. In fact we need a quantitative version and an extension to higher dimensions of Orponen's argument. Since we follow Orponen's ideas quite closely, the proofs of these facts are deferred to Appendix \ref{sec:spherical-projections}.

\subsection{Structure of the paper and further results}

In Section \ref{sec:notation-and-preliminary} we introduce some general notation, and some preliminary lemmas. In particular, in \S\ref{subsec:regular-measures} we review the important concept of regular measures, and their properties, and in \S\ref{subsec:robust-measures}, we and introduce the convenient notions of robust measures and robust entropy, and relate them to Hausdorff dimension.

Section \ref{sec:projection-theorems} deals with discretized projection theorems. The main result of this section is Theorem \ref{thm:projection-combined} which, in the language of robust measures, combines and unifies Bourgain's discretized projection theorem with quantitative versions of more classical projection theorems going back to R.~Kaufman and K.~Falconer.

Section \ref{sec:multiscale-decompositions} contains two new multiscale decompositions of (regular) measures, see Theorems \ref{thm:multiscale-decomposition} and \ref{thm:multiscale-decomposition-Ahlfors}. The proofs are reduced to combinatorial statements about Lipschitz functions, that take up most of the section.

Theorems \ref{thm:main-continuous} and \ref{thm:main-discrete} are proved in Section \ref{sec:proofs}.

In Section \ref{sec:applications} we derive several applications and generalizations of Theorems \ref{thm:main-continuous} and \ref{thm:main-discrete}. We begin in \S\ref{subsec:real-analytic} with a straightforward application to one-parameter real-analytic families. In \S\ref{subsec:distance-sets}, we complete the proof of Theorem \ref{thm:main-distance-sets}. In \S\ref{subsec:incidences}, we apply Theorem \ref{thm:main-discrete} to obtain incidence bounds for discretized families of curves, see Theorem \ref{thm:incidences}. This result is new even for lines; in this particular case, it complements results of M.~Bateman and V.~Lie \cite{BatemanLie19}, and of L.~Guth, N.~Solomon and H.~Wang \cite{GSW19}.   Note that Theorem \ref{thm:main-continuous} only applies to projections onto the real line. In \S\ref{subsec:higher-rank}, we indicate how to use W.~He's extension of Bourgain's projection theorem to higher rank projections in order to derive a higher rank extension of Theorem \ref{thm:main-continuous}, see Theorem \ref{thm:main-continuous-higher-rank}. Finally, we use this higher rank version to study spherical projections in \S\ref{subsec:spherical}, in particular extending recent results of T.~Orponen, and of B.~Liu and C-Y.~Shen.

As indicated earlier, Appendix \ref{sec:entropy-of-projections} contains the proof of  Proposition \ref{prop:good-bound-from-below}, a crucial estimate of the entropy of a smooth image of a measure in terms of linearized images of small pieces of the measure in a multiscale decomposition. These are a key ingredient in the proofs of Theorems \ref{thm:main-continuous} and \ref{thm:main-discrete}. The ideas in this appendix are not new, but the statements differ enough from existing ones in the literature that we have chosen to include a complete proof. Likewise, Appendix \ref{sec:spherical-projections} contains the proof of Theorem \ref{thm:radial-positive-dim}, a  higher dimensional variant of a spherical projection result of T.~Orponen \cite{Orponen19}. The results in this appendix are used in the proofs of Theorem \ref{thm:main-distance-sets} in \S\ref{subsec:distance-sets} and of Theorem \ref{thm:radial-projections-dim} in \S\ref{subsec:spherical}. Although we follow Orponen's ideas closely, there are several changes in the details, so again we include full details for completeness.

\subsection*{Acknowledgements} Part of this work was completed while the author was visiting the Universities of St Andrews and Cambridge. I am grateful for hospitality and a productive work environment at both places. I thank K.~H\'{e}ra for pointing out Problem 4 in \cite{Ostrava10}. I am grateful to V.~Lie for useful comments, and for sharing the results from \cite{BatemanLie19} . I also have J.~Zahl to thank for useful insights, and in particular for pointing out the relevance of ``train-track'' examples in incidence counting. Finally, I thank an anonymous referee for helpful comments that helped improve the presentation.

\section{Notation and preliminary results}
\label{sec:notation-and-preliminary}

\subsection{Notation}
\label{subsec:notation}

We use Landau's $O(\cdot)$ notation: given $X>0$, $O(X)$ denotes a positive quantity bounded above by $C X$ for some constant $C>0$. If $C$ is allowed to depend on some other parameters, these are denoted by subscripts. We sometimes write $X\lesssim Y$ in place of $X=O(Y)$ and likewise with subscripts. We write $X\gtrsim Y$,  $X\approx Y$ to denote $Y\lesssim X$,  $X\lesssim Y\lesssim X$ respectively.

The family of Borel probability measures on a metric space $X$ is denoted by $\cP(X)$, and the family of Borel finite measures by $\cM(X)$.

Logarithms are always to base $2$.

We let $\cD_j$ be the family of half-open $2^{-j}$-dyadic cubes in $\R^d$ (where $d$ is understood from context), and let $\cD_j(x)$ be the only cube in $\cD_j$ containing $x\in \R^d$. Given a measure $\mu\in\cP(\R^d)$, we also let $\cD_j(\mu)$ be the cubes in $\cD_j$ with positive $\mu$-measure. We recall that, given $A\subset \R^d$, we also denote by $\cN(A,j)$ the number of cubes in $\cD_j$ that intersect $A$.

A $2^{-m}$-measure is a measure $\mu$ in $\cP([0,1)^d)$ such that $\mu_Q$ is a multiple of Lebesgue measure on $Q$ for $Q\in\cD_m$. Hence, $2^{-m}$-measures are defined down to resolution $2^{-m}$. The set of $2^{-m}$ measures on $\R^d$ will be denoted $\cP_m^d$. Likewise, a $2^{-m}$-set is a union of cubes in $\cD_m$.

Due to our use of dyadic cubes,  sometimes we will need to deal with supports in the dyadic metric, i.e. given $\mu\in\cP([0,1)^d)$ we let
\[
\supp_{\mathsf{d}}(\mu) = \{ x: \mu(\cD_j(x))>0 \text{ for all } j\in\N\}.
\]
Note that $\mu(\supp_{\mathsf{d}}(\mu))=1$ and that $\supp_{\mathsf{d}}(\mu)\subset \supp(\mu)$.

If a measure $\mu\in\cP(\R^d)$ has a density in $L^p$, then its density is sometimes also denoted by $\mu$, and in particular  $\|\mu\|_p$ stands for the $L^p$ norm of its density.

Let $\mu\in\cP([0,1)^d)$. Recall that if $\mu(A)>0$, then $\mu_A=\tfrac{1}{\mu(A)}\mu|_A$ is the normalized restriction of $\mu$ to $A$. If $Q$ is a dyadic cube and $\mu(Q)>0$, then we denote $\mu^Q = \text{Hom}_Q\mu_Q$, where  $\text{Hom}_Q$ is the homothety renormalizing $Q$ to $[0,1)^d$. Thus, $\mu^Q$ is a magnified and renormalized copy of the restriction of $\mu$ to $Q$; we sometimes refer to such measures as \emph{conditional measures} on $Q$.

Recall that the Grassmanian of linear $k$-planes in $\R^d$ is denoted by $\mathbb{G}(d,k)$. When $k=1$, we often identify $\mathbb{G}(d,1)$ with $S^{d-1}$; the fact that the identification is two-to-one does not cause any issues in practice. We denote the manifold of \emph{affine} $k$-planes in $\R^d$ by $\mathbb{A}(d,k)$.

\subsection{Regular measures}
\label{subsec:regular-measures}

A key role in the paper is played by measures with a uniform tree (or Moran) structure when represented in base $2^T$. This notion is made precise in the next definition. Recall that $\cP_m^d$ stands for the family of $2^{-m}$-measures in $\R^d$.
\begin{definition} \label{def:regular}
Given a sequence $\sigma=(\sigma_1,\ldots,\sigma_{\ell})\in [0,d]^\ell$ and $T\in\N$, we say that $\mu\in \cP_{T\ell}^d$ is \emph{$(\sigma;T)$-regular} if for any $Q\in \cD_{jT}(\mu)$, $1\le j\le\ell$, we have
\[
\mu(Q) \le 2^{-T \sigma_j} \mu(\wh{Q}) \le 2\mu(Q),
\]
where $\wh{Q}$ is the only cube in $\cD_{(j-1)T}$ containing $Q$. When $T$ is understood from context we will simply write that $\mu$ is $\sigma$-regular. The family of $((\sigma_1,\ldots,\sigma_\ell);T)$-regular measures on $\R^d$ will be denoted by $\cR_{T,\ell}^d$.
\end{definition}
We note that the same notion appears in \cite{KeletiShmerkin19}, but with a different normalization for the parameters $\sigma_j$.

\begin{lemma} \label{lem:frostman-regular}
Let $\nu\in\cP([0,1)^d)$ be $(\sigma;T)$-regular for some $\sigma\in [0,d]^\ell$, $T\in\N$. Write $m=T\ell$ and $X=\supp_{\mathsf{d}}(\nu)$.
\begin{enumerate}[(\rm i)]
  \item \label{it:frostman-regular-i}
  \[
2^{-j} 2^{- T(\sigma_1+\ldots+\sigma_j)} \le \nu(Q) \le 2^{- T(\sigma_1+\ldots+\sigma_j)}
\]
for all $Q\in\cD_{jT}(X)$.
\item  \label{it:frostman-regular-ii}
\[
2^{T(\sigma_1+\ldots+\sigma_\ell)} \le \cN(X,m) \le  2^{T(\sigma_1+\ldots+\sigma_\ell)+\ell}.
\]
\item  \label{it:frostman-regular-iii}
\[
  2^{-\ell} \nu \le \mathbf{1}_X/|X| \le 2^{\ell}\nu.
\]
\end{enumerate}
\end{lemma}
\begin{proof}
From the definition it is clear that if $Q\in\cD_{jT}(\nu)=\cD_{jT}(X)$, then
\begin{equation} \label{eq:regular-measure-comparison}
2^{-j} 2^{-T\sigma_1}  \cdots 2^{-T\sigma_j}\le \nu(Q) \le 2^{-T\sigma_1}  \cdots 2^{-T\sigma_j},
\end{equation}
which is \eqref{it:frostman-regular-i}. Claim \eqref{it:frostman-regular-ii} follows easily from \eqref{eq:regular-measure-comparison} applied with $j=\ell$. For the final claim, it is enough to establish the inequality for $Q\in\cD_m$, and this follows from \eqref{eq:regular-measure-comparison} (which implies $\nu(Q) \le 2^{m/T}\nu(Q')$ for $Q,Q'\in\cD_m$) and \eqref{it:frostman-regular-ii}.
\end{proof}

Starting with an arbitrary $2^{-m}$-measure, a pigeonholing argument (which we learned from Bourgain's work) allows us to find a set $X$ with ``large'' measure such that $\mu_X$ is regular:

\begin{lemma} \label{lem:subset-regular}
Fix $T,\ell\ge 1$. Write $m=T\ell$, and let $\mu\in\cP_m^d$. Then there is a set $X\subset\supp_{\mathsf{d}}(\mu)$ with $\mu(X)\ge (2dT+2)^{-\ell}$ such that $\mu_X$ is $(\sigma;T)$-regular for some $\sigma\in [0,d]^\ell$.
\end{lemma}
See \cite[Lemma 3.4]{KeletiShmerkin19} for the proof of this particular statement. Iterating the above lemma, we can decompose an arbitrary measure $\mu\in\cP_m^d$ into regular measures, plus a negligible error. This is the content of the next crucial lemma whose proof can be found in \cite[Corollary 3.5]{KeletiShmerkin19}.
\begin{lemma} \label{lem:decomposition-regular}
Fix $T,\ell\ge 1$ and $\e>0$. Write $m=T\ell$, and let $\mu\in\cP_m^d$.  There exists a family of pairwise disjoint $2^{-m}$-sets $X_1,\ldots, X_N$  with $X_k\subset\supp_{\mathsf{d}}(\mu)$, and such that:
\begin{enumerate}[(\rm i)]
\item $\mu\left(\bigcup_{k=1}^N X_k\right) \ge 1-2^{-\e m}$. In particular, if $\mu(A)> 2^{-\e m}$, then there exists $k$ such that $\mu_{X_k}(A)\ge \mu(A)-2^{-\e m}$.
\item $\mu(X_k) \ge 2^{-\delta m}$, where $\delta=\e+\log(2d T+2)/T$,
\item Each $\mu_{X_k}$ is $(\sigma(k);T)$-regular for some $\sigma(k)\in  [0,d]^\ell$.
\end{enumerate}

\end{lemma}

\subsection{Robust measures and robust entropy}
\label{subsec:robust-measures}

We will need to deal with various different notions of ``largeness'' of a measure. The next definition captures the core property shared by all of these notions.
\begin{definition} \label{def:measure-robust}
A a measure $\mu\in\cP(\R^d)$ is called \emph{$(\alpha,\delta,m)$-robust} if, for any set $A$ with $\mu(A) > 2^{-\delta m}$, one has $\mathcal{N}(A,m) > 2^{\alpha m}$.
\end{definition}

\begin{lemma} \label{lem:robust-to-dimension}
If $\mu\in\cP(\R^d)$ is $(\alpha,\delta,m)$-robust for all $m\ge m_0$, then $\hdim(A)\ge \alpha$ for all Borel sets $A$ with $\mu(A)>0$. Furthermore, if $\{ B_j\}$ is a union of disjoint balls of radius $2^{-m}$ for some $m\ge m_0$ such that $\mu(B_j) \ge 2^{-\alpha m}$, then
\[
\mu(\cup_j B_j) \le 2^{-\delta m}
\]
\end{lemma}
\begin{proof}
Since there can be at most $2^{\alpha m}$ disjoint balls of measure $\ge 2^{-\alpha m}$, the second claim is immediate from the definition of robustness. It follows by a standard covering lemma argument and Borel-Cantelli that for $\mu$-almost all $x$, $\mu(B(x,r)) \le r^\alpha$ for all sufficiently small $r$ (depending on $x$). The first claim is now a consequence of the mass distribution principle.
\end{proof}

Recall that the entropy of $\mu\in\cP(X)$ with respect to a finite partition $\cA$ of $X$ (or of a set of full $\mu$-measure in $X$) is defined by
\[
H(\mu,\cA) = \sum_{A\in\cA} \mu(A)\log(1/\mu(A)).
\]
We will sometimes write $H_m(\mu)$ in place of $H(\mu,\cD_m)$. Note that this quantity is \emph{not} normalized. For more about entropy, see Appendix \ref{sec:entropy-of-projections}.

Given two measures $\mu,\nu$ on the same space $X$ and a number $\Delta>0$, we write $\nu\le\Delta\mu$ if $\nu(A)\le \Delta\mu(A)$ for all Borel sets $A$. In particular,  $\mu_B\le \mu(B)^{-1}\mu$.
\begin{definition}
Let $\mu\in\cP(\R^d)$, fix $\Delta\ge 1$, and let $\mathcal{A}$ be a finite partition of $\supp(\mu)$. We define the $\Delta$-robust entropy $H^{\Delta}(\mu,\mathcal{A})$ as
\[
\inf \{H(\nu,\mathcal{A}): \nu\in\cP(\R^d), \nu\le\Delta \mu\}.
\]
In the case $\mathcal{A}=\mathcal{D}_m$, we sometimes write $H_m^\Delta(\mu)$ in place of $H^\Delta(\mu,\mathcal{D}_m)$.
\end{definition}

The next lemma asserts that robust measures have large robust entropy.
\begin{lemma} \label{lem:robust-to-entropy}
Given $\alpha,\delta,\e>0$, the following holds for all $m\ge m_0(\delta,\e)$: if $\mu\in\cP(\R^d)$ is $(\alpha,\delta,m)$-robust, then
\[
H^{2^{\delta m/2}}_m(\mu) \ge  (\alpha-\e)m.
\]
\end{lemma}
\begin{proof}
It follows from the definition that if $\nu \le 2^{\delta m/2} \mu$, then $\nu$ is $(\alpha,\delta/2,m)$-robust. So it is enough to show that if $\nu$ is $(\alpha,\delta',m)$-robust and $m$ is sufficiently large (depending on $\alpha,\delta',\e$), then
\[
H(\nu,\cD_m) \ge (\alpha-\e)m.
\]
Now, if $\nu$ is $(\alpha,\delta',m)$-robust, then the $\nu$-mass of the union of all the cubes in $\cD_m$ of $\nu$-measure $> 2^{-\alpha m}$ is $\le 2^{-\delta' m}\le 1/2$ (as there are $<2^{\alpha m}$ such cubes). If $A$ denotes the union of all the cubes in $\cD_m$ of $\nu$-mass $\le 2^{-\alpha m}$, then $\nu_A(Q) \le 2^{1-\alpha m}$ for $Q\in\cD_m$, and therefore, by the concavity of entropy,
\[
H(\nu,\cD_m) \ge \nu(A)H(\nu_A,\cD_m) \ge  (1-2^{-\delta' m})(\alpha m-1) \ge (\alpha-\e)m,
\]
provided $m$ is large enough in terms of $\delta',\e$.
\end{proof}


\section{Discretized projection theorems}
\label{sec:projection-theorems}

Given $\theta\in S^{d-1}$, we denote the orthogonal projection $x\mapsto \langle \theta, x\rangle $, $\R^d\mapsto \R$, by $P_\theta$. If $\mu$ is a measure on $\R^d$, we also write $\mu_\theta= P_\theta\mu$.

We begin by recalling Bourgain's discretized projection theorem. There are several equivalent variants of the statement; the one we state is the special case $m=1$ from \cite[Theorem 1]{He20}.
\begin{theorem} \label{thm:projection-Bourgain}
Given $0<\alpha<d$ and $0<\kappa<1$ there exist $\delta,\eta>0$ such that the following holds for all sufficiently large $m$. Let $X\subset B^d(0,1)$, and let $\rho\in\cP(S^{d-1})$ satisfy
\begin{align*}
\cN(X,m) &\ge 2^{m(\alpha-\delta)},\\
\cN(X\cap B(x,r),m) &\le 2^{\delta m} r^\kappa \cN(X,m) \quad\text{for all } r\in [2^{-m},1], x\in B^d(0,1),\\
\rho(H^{(r)}\cap S^{d-1}) &\le 2^{\delta m} r^\kappa \quad\text{for all } r\in [2^{-m},1], H\in \mathbb{G}(d,d-1).
\end{align*}
Then there is a set $E\subset S^{d-1}$ with $\rho(E)\le 2^{-\delta m}$ such that if $\theta\in S^{d-1}\setminus E$ and  $X'\subset X$ satisfies $\cN(X',m) \ge 2^{-\delta m}\cN(X,m)$, then
\[
\log\cN(P_\theta X',m) \ge m\left(\frac{\alpha}{d}+\eta\right).
\]
\end{theorem}
Note that if we make $\eta$ and $\delta$ slightly smaller, then they also work for nearby values of $\kappa$ and $\alpha$. Hence there is no loss of generality in assuming that $\eta$ and $\delta$ are continuous functions of $(\kappa,\alpha)$.

While Bourgain's projection theorem will be our main tool, we will also need to consider the case in which $\log\cN(A,m)/m$ is close to either $0$ or $d$. Crucially, we need $\delta$ to be independent of $\log\cN(A,m)$ in the estimates; clearly, in Theorem \ref{thm:projection-Bourgain} the value of $\delta$ must depend on $\alpha$ if we allow values of $\alpha$ close to $0$ or $d$. Hence we need to revisit other (more classical) projection theorems, and combine them into a single result that deals with the three regimes which $\log\cN(A,m)/m$ is close to $0$, close to $d$ or far from both $0$ and $d$ . We start with a result that essentially goes back to R.~Kaufman \cite{Kaufman68} in the 1960s. Before stating it, recall that the $\sigma$-energy of $\nu\in\cM(\R^d)$ is defined as
\[
\cE_\sigma(\nu)= \iint |x-y|^{-\sigma} \,d\nu(x)\,d\nu(y) \in (0,\infty].
\]
\begin{theorem} \label{thm:projection-Kaufman}
Fix $0<\sigma<\kappa<1$. Let $\rho\in \mathcal{P}(S^{d-1})$ satisfy
\[
\rho(H^{(r)}\cap S^{d-1}) \le C r^{\kappa}, \quad H\in \mathbb{G}(d,d-1), r>0,
\]
and fix $\mu\in\cP(\R^d)$. Then
\[
\int_{S^{d-1}} \cE_{\sigma}(\mu_\theta)\, d\rho(\theta) \le \left(1+\frac{C\sigma}{\kappa-\sigma}\right) \cE_{\sigma}(\mu).
\]
\end{theorem}
\begin{proof}
The claim is implicit in the proof of \cite[Theorem 5.1]{Mattila15}; since it is not explicitly stated in this form, we repeat the short argument for completeness.  Fix $x\in\R^d\setminus\{0\}$, and note that
\[
\{ \theta\in \R^d: |P_\theta(x)|< \delta \} = x^{\perp}+B(0,\delta/|x|),
\]
and hence, using the assumption on $\rho$,
\[
\rho\{\theta\in S^{d-1}: |P_\theta(x)|<\delta \} \le C (\delta/|x|)^\kappa.
\]
Using this and Fubini, we estimate
\begin{align*}
\int_{S^{d-1}} |P_\theta(x)|^{-\sigma}\,d\rho(\theta) &= \int_0^\infty \rho\{\theta: |P_\theta(x)|^{-\sigma}> r\}\,dr\\
&= \int_0^{|x|^{-\sigma}} 1 \,dr + \int_{|x|^{-\sigma}}^\infty  \rho\{\theta: |P_\theta(x)|< r^{-1/\sigma}\}\,dr\\
&\le  \left(1 + \frac{C\sigma}{\kappa-\sigma}\right) |x|^{-\sigma}.
\end{align*}
Using Fubini, we conclude
\begin{align*}
\int_{S^{d-1}} \cE_{\sigma}(\mu_\theta)\, d\rho(\theta) &=  \int_{S^{d-1}} \iint_{\R^d\times\R^d} |P_\theta(x-y)|^{-\sigma}\, d\mu(x)d\mu(y) d\rho(\theta)\\
&\le \left(1 + \frac{C\sigma}{\kappa-\sigma}\right) \iint_{\R^d\times\R^d} |x-y|^{-\sigma} \, d\mu(x)d\mu(y),
\end{align*}
as claimed.
\end{proof}

The following is a quantitative form of K.~Falconer's classical bound on the dimension of exceptional projections \cite{Falconer82}.
\begin{theorem} \label{thm:projection-Falconer}
Let $\rho\in \mathcal{M}(S^{d-1})$ satisfy
\[
\rho(B(\theta,r)) \le r^{\kappa}, \quad \theta\in S^{d-1}, r>0,
\]
and fix $\mu\in\cP(\R^d)$ such that $\cE_{d-\kappa}(\mu)<\infty$. Then $\mu_\theta\in L^2$ for $\rho$-almost all $\theta$, and
\[
\int_{S^{d-1}}  \| \mu_\theta\|_2^2 \, d\rho(\theta) \lesssim_d \cE_{d-\kappa}(\mu).
\]
\end{theorem}
\begin{proof}
In the course of the proof of \cite[Theorem 5.6]{Mattila15} it is shown that
\[
\int_{S^{d-1}} \int_{\R^d} |\widehat{\mu}_\theta(x)|^2 \, dx\, d\rho(\theta) \lesssim_d \int_{\R^d} |\widehat{\mu}(x)|^2 (1+|x|)^{-\kappa}\,dx.
\]
Note that our $\kappa$ corresponds to $\tau$ in \cite[Theorem 5.6]{Mattila15}, and $t=1$ in our setting. The well-known expression of the energy in terms of the Fourier transform (see e.g. \cite[Theorem 3.10]{Mattila15}) together with Plancherel yields the result.
\end{proof}

To conclude this section, we combine Theorems \ref{thm:projection-Bourgain}, \ref{thm:projection-Kaufman} and \ref{thm:projection-Falconer} in the language of robust measures.
\begin{theorem} \label{thm:projection-combined}
Given $0<\kappa<1$ there exists $\eta=\eta(\kappa)>0$ such that the following holds for all sufficiently small $\delta\le \delta_0(\kappa)$ and all sufficiently large $m\ge m_0(\delta)$. Fix $\alpha\in [0,d]$.  Let $\mu\in\cP_m^d$, and let $\rho\in\cP(S^{d-1})$ satisfy
\begin{align*}
\mu(B(x,r)) &\le 2^{\delta m} r^\alpha \quad\text{for all } r\in [2^{-m},1], x\in [0,1)^d,\\
\rho(H^{(r)}\cap S^{d-1}) &\le 2^{\delta m} r^\kappa \quad\text{for all } r\in [2^{-m},1], H\in \mathbb{G}(d,d-1).
\end{align*}

Then there is a set $E\subset S^{d-1}$ with $\rho(E)\le 2^{-\delta m}$ such that $\mu_\theta$ is $(\gamma(\alpha),\delta,m)$-robust for all $\theta\in S^{d-1}\setminus E$, where
\[
\gamma(\alpha) = \left\{
\begin{array}{ccc}
  \alpha-6\delta & \text{ if } & \alpha < \kappa/2 \\
  \alpha/d+\eta & \text{ if } & \kappa/2 \le \alpha \le d-\kappa/2 \\
  1-6\delta & \text{ if } &  \alpha > d-\kappa/2
\end{array}
\right..
\]
\end{theorem}
\begin{proof}
\textbf{First case: $\alpha\in [\kappa/2,d-\kappa/2]$}: As already noted, we may assume that in Theorem \ref{thm:projection-Bourgain} the same values of $\delta$ and $\eta$ work for all $\alpha\in [\kappa/2, d-\kappa/2]$ (hence $\delta,\eta$ depend on $\kappa$ only). We will in fact assume that $\delta=\delta(\kappa)$ is small enough that the conclusion of Theorem \ref{thm:projection-Bourgain} holds for $4\delta$ in place of $\delta$.

Theorem \ref{thm:projection-Bourgain} applies to sets rather than measures, as in our current context. We will use a fairly standard argument to deal with this. For an integer $j\ge 0$, let
\[
X_j = \bigcup\{ Q\in\cD_m: 2^{-j-1}<\mu(Q)\le 2^{-j} \},
\]
and set $J=\{ j: \mu(X_j)\ge 2^{-2\delta m}\}$. Note that
\[
\mu\left(\bigcup_{j=2dm}^\infty X_j\right) \le 2^{dm} 2^{-2dm} = 2^{-dm}< 2^{-2\delta m}.
\]
In particular, $J\subset \{0,\ldots,2dm-1\}$, and if we set $Z=[0,1)^d\setminus \cup_{j\in J} X_j$, then
\begin{equation} \label{eq:bad-set-small-measure}
\mu(Z) \le 2dm 2^{-2\delta m}+ 2^{-dm} \le 3dm 2^{-2\delta m}.
\end{equation}
Now the non-concentration condition on $\mu$ implies that, for $j\in J$,
\[
\mu_{X_j}(B(x,r)) \le \mu(X_j)^{-1} \mu(B(x,r)) \le 2^{3\delta m} r^{\alpha} \quad( x\in [0,1)^d, r\in [2^{-m},1]).
\]
In particular,
\[
\mu_{X_j}(Q) \lesssim 2^{(3\delta-\alpha) m} \quad\text{for all }Q\in\cD_m,
\]
so that $\cN(X_j,m) \gtrsim 2^{(\alpha-3\delta)m}$. Also, by definition, $\mathbf{1}_{X_j}/|X_j| \le 2 \mu_{X_j}$. Hence
\[
|X_j\cap B(x,r)| \lesssim 2^{3\delta m}|X_j| r^\alpha   \quad( x\in [0,1)^d, r\in [2^{-m},1]),
\]
and since $X_j$ is a union of $2^{-m}$-cubes, this translates into a corresponding bound for the counting numbers $\cN(X_j\cap B(x,r),m)$.

We have checked that the hypotheses of Theorem \ref{thm:projection-Bourgain} hold for $X_j$, $j\in J$ (with $4\delta$ in place of $\delta$). Let $E_j$ be the exceptional set given by the theorem, and define $E=\cup_{j\in J} E_j$. Note that
\[
\rho(E)\le |J| 2^{-4\delta m} \le 2dm 2^{-4\delta m} < 2^{-\delta m}.
\]
Fix $\theta\in S^{d-1}\setminus E$, and let $A$ be a set with $P_\theta\mu(A)\ge 2^{-\delta m}$. It follows from \eqref{eq:bad-set-small-measure} and the decomposition $\mu=\mu(Z)\mu_Z+\sum_{j\in J} \mu(X_j)\mu_{X_j}$ that there is $j\in J$ such that
\[
\mu_{X_j}(P_\theta^{-1}A)\ge \mu(P_\theta^{-1} A)-\mu(Z) \ge \tfrac{1}{2} 2^{-\delta m},
\]
and hence $|X_j\cap P_\theta^{-1}A|\ge \tfrac{1}{4}2^{-\delta m}|X_j|$. Again using the fact that $X_j$ is a union of cubes in $\cD_m$, we get a corresponding estimate for counting numbers. Since $\theta\notin E_j$, Theorem \ref{thm:projection-Bourgain} implies that
\[
\log \cN(A, m) \ge (\alpha/d+\eta)m,
\]
and thus we have verified that $P_\theta\mu$ is $(\alpha/d+\eta,\delta,m)$-robust, as desired.

\textbf{Second case: $\alpha <\kappa/2$}: It follows from the non-concentration assumption on $\mu$ that
 \begin{align*}
 \cE_{\alpha}(\mu)  &\lesssim \sum_{p=0}^{m-1} 2^{p\alpha} (\mu\times\mu) \{ (x,y): |x-y|\le 2^{1-p} \} \\
 &\le \sum_{p=0}^{m-1} 2^{p\alpha}  \max_x\{ \mu(B(x,2^{1-p})) \} \lesssim m 2^{\delta m}.
 \end{align*}
 Hence we get from Theorem \ref{thm:projection-Kaufman} (applied with $C=2^{\delta m}$ and $\alpha$ in place of $\sigma$) that
 \[
 \int_{S^{d-1}} \cE_\alpha(\mu_\theta)\,d\rho(\theta) \lesssim_{\kappa} m 2^{2\delta m},
 \]
and therefore $\rho(E) \le 2^{-\delta m}$, where
 \[
 E = \{ \theta\in S^{d-1}: \cE_\alpha(\mu_\theta) \ge C_\kappa m 2^{3\delta m}\},
 \]
 for a suitable $C_\kappa>0$. Fix $\theta\in S^{d-1}\setminus E$ and suppose $\mu_\theta(A)\ge 2^{-\delta m}$. Then, writing $\nu=(\mu_\theta)_A$,
 \[
 \cE_\alpha(\nu) \lesssim \mu_\theta(A)^{-2} \cE_\alpha(\mu_\theta) \le C_\kappa m 2^{5\delta m}.
  \]
 On the other hand, it follows e.g. from \cite[Lemma 3.1]{KeletiShmerkin19} and Cauchy-Schwarz that
  \[
  \cE_\alpha(\nu) \gtrsim_{\alpha} 2^{\alpha m}\sum_{I\in\cD_m} \nu(I)^2 \ge 2^{\alpha m}\cN(A,m)^{-1}.
  \]
  Combining the last two displayed equations we see that $\mu_\theta$ is $(\alpha-6\delta,\delta,m)$-robust if $m$ is large enough, as claimed.

\textbf{Third case: $\alpha >d-\kappa/2$}. To begin, we note that, arguing as above and using the non-concentration assumption on $\mu$,
\[
 \cE_{d-\kappa}(\mu) \lesssim \sum_{p=0}^{m-1} 2^{p(d-\kappa)}  \max_x\{ \mu(B(x,2^{1-p})) \} \lesssim_\kappa 2^{\delta m}.
\]
Applying Theorem \ref{thm:projection-Falconer} to $2^{-\delta m}\rho$, we deduce that, provided $m$ is large enough,
\[
\int_{S^{d-1}} \|\mu_\theta\|_2^2 \,d\rho(\theta) \lesssim 2^{2\delta m}.
\]
Let $E=\{ \theta: \|\mu_\theta\|_2^2 \ge C_d 2^{3\delta m}\}$, where $C_d$ is large enough, so that $\rho(E)\le 2^{-\delta m}$. Fix $\theta\in S^{d-1}\setminus E$ and $A$ such that $\mu_\theta(A)\ge 2^{-\delta m}$. Writing again $\nu=(\mu_\theta)_A$, we have $\|\nu\|_2^2 \le \mu_\theta(A)^{-2}\|\mu_\theta\|_2^2 \le C_d 2^{5\delta m}$, and by a well-known application of Cauchy-Schwarz (see e.g. \cite[Lemma 6.5]{KeletiShmerkin19}) we conclude that $\cN(A,m) \ge 2^{m} \|\nu\|_2^{-2} \ge 2^{(1-6\delta)m}$, confirming that $\mu_\theta$ is $(1-6\delta,\delta,m)$-robust.
\end{proof}

\section{Multiscale decompositions of regular measures}
\label{sec:multiscale-decompositions}

\subsection{A new multiscale decomposition}

The goal of this section is to establish Theorem \ref{thm:multiscale-decomposition}, providing a new kind of multiscale decomposition of a regular measure $\mu$. Recall that, by Lemma \ref{lem:decomposition-regular}, one can decompose an arbitrary $2^{-m}$-measure into regular pieces plus an error term so, as we will see, this decomposition is also useful to study general measures.

Roughly speaking, the conclusion of the theorem says that given a regular measure $\mu$, one can find a sequence of scales $2^{-m_j}$, such that for $Q\in\cD_{m_j}$, the conditional measures $\mu^Q$ satisfy a near-Frostman decay condition. Moreover, and crucially, for a positive density of scales (weighted according to the measure), the Frostman exponent is bounded away from $0$ and $1$. This last claim fails if $\mu$ is the uniform measure on a square, and the assumptions of the theorem are meant precisely to avoid this counterexample. Moreover, the scales $B_j$ can be chosen to satisfy $B_{j+1}\le 2B_j$, which is critical for linearization arguments.
\begin{theorem} \label{thm:multiscale-decomposition}
For every $u>0$ and $\e>0$ there are $\xi=\xi(u)>0$ and $\tau=\tau(\e)>0$ such that the following holds for all sufficiently large $T\ge T_0(\e)$ and  $\ell\ge \ell_0(T,\e)$:

Let $\mu$ be a $(\sigma;T)$-regular measure on $[0,1)^d$ with (dyadic) support $X$, and write $m=\ell T$. Suppose
\begin{equation} \label{eq:single-scale-Frostman}
\mu(B(x,|X|^{1/d})) \le 2^{-u m}\quad \text{for all }x.
\end{equation}
Then there are a collection of pairwise disjoint intervals $\{[A_j,B_j)\}$ contained in $[0,\ell)$ and numbers $\alpha_j\in [0,d]$, such that the following hold:
\begin{enumerate}[(\rm i)]
  \item \label{it:multiscale-i} $ \tau \ell \le B_j - A_j \le A_j$ for all $j$.
  \item \label{it:multiscale-ii}  Write $m_j= T(B_j-A_j)$. For each $Q\in \cD_{T A_j}$,
\[
  \mu^Q(B(x,r)) \le 2^{\e m_j} r^{\alpha_j}\quad\text{for all } r\in [2^{-m_j},1].
\]
  \item \label{it:multiscale-iii}
  \[
  \sum_j \alpha_j m_j \ge \log \cN(X,m) - 2\e m. 
  \]
  \item \label{it:multiscale-iv}
  \[
  \sum\{ m_j: \alpha_j \in [\xi,d-\xi]\} \ge \xi m.
  \]
\end{enumerate}
\end{theorem}

Theorem \ref{thm:multiscale-decomposition} will be proved in the rest of the section. Following \cite{KeletiShmerkin19}, the problem is translated into one about Lipschitz functions on the line, but both the statement of the problem and the solution differ substantially from \cite{KeletiShmerkin19}.

\subsection{Decompositions of Lipschitz functions into almost linear/superlinear pieces}

In this section we deal with the following kind of problem: given a Lipschitz function $f:[a,b]\to\R$, we aim to find non-overlapping intervals $I_j$ such that $f|_{I_j}$ is close to linear/bounded below by a linear function, and the union of the $I_j$ exhaust most of the original interval $[a,b]$. We start by making some of these concepts precise.
\begin{definition}
Given a function $f:[a,b]\to\R$, we let
\[
s_f(a,b) = \frac{f(b)-f(a)}{b-a}
\]
be the slope of the linear function that agrees with $f$ on $a$ and $b$. We also write
\[
L_{f,a,b}(x) = f(a) +s_f(a,b)(x-a)
\]
for the linear function that agrees with $f$ at $a$ and $b$. We say that \emph{$(f,a,b)$ is $\e$-linear} if
\[
\big|f(x)-  L_{f,a,b}(x) \big|\le \e |b-a|\quad\text{for all } x\in [a,b].
\]
 Likewise, we say that \emph{$(f,a,b)$ is $\e$-superlinear} if
\[
f(x) \ge L_{f,a,b}(x) - \e|b-a| \quad\text{for all } x\in [a,b].
\]
Sometimes we say that $f$ is linear/superlinear on $[a,b]$ to mean that $(f,a,b)$ is linear/superlinear.
\end{definition}

The following is our basic lemma for finding intervals on which $f$ is $\e$-linear. A small variant of the lemma was posed as a problem on the $20$-th Annual {V}ojt\v{e}ch {J}arn\'{\i}k International Mathematical Competition; I thank K.~H\'{e}ra for pointing this out. The proof below is repeated from \cite{Ostrava10}.

\begin{lemma} \label{lem:tube-null-1}
For every $\e>0$ there is $\delta=\delta(\e)>0$ such that the following holds: for any $1$-Lipschitz function $f:[a,b]\to \R$ there exists a sub-interval $[c,d]\subset [a,b]$ with $(d-c)\ge \delta (b-a)$ such that $(f,c,d)$ is $\e$-linear. In fact, $\delta=\e^{\lfloor 1/\e\rfloor }$ works.
\end{lemma}
\begin{proof}
By replacing $f$ with $-f$ if needed, we may assume that $f(b)\ge f(a)$. We claim that if $(f,a,b)$ is \emph{not} $\e$-linear, then there exists an interval $[a',b']\subset [a,b]$ with $b'-a' \ge \e(b-a)$ such that $s_f(a',b') \ge s_f(a,b)+\e$. Suppose, then, that $(f,a,b)$ is not $\e$-linear, which by definition means that there is $x\in [a,b]$ such that
\[
\big|f(x)-  L_{f,a,b}(x) \big|> \e (b-a).
\]
Replacing, if needed, $f$ by the flip $\wt{f}(x)=-f(a+b-x)$ and $x$ by $a+b-x$, we may assume that $f(x)-L_{f,a,b}(x) > \e(b-a)$. We have
\[
s_f(a,x) = \frac{f(x)-f(a)}{x-a} \ge \frac{s_f(a,b)(x-a) + \e(b-a)}{x-a} \ge s_f(a,b)+\e.
\]
On the other hand,
\[
x-a \ge f(x)-f(a) \ge  s_f(a,b)(x-a) + \e(b-a) \ge \e(b-a),
\]
so $[a,x]$ is the claimed interval.

Now let $[a_0,b_0]=[a,b]$, and inductively apply the claim and set $[a_{j+1},b_{j+1}]=[a'_j,b'_j]$ so long as $(f,a_j,b_j)$ is not $\e$-linear. Since $s_f(a_0,b_0)\ge 0$ and $s_f(a_j,b_j)\le 1$, the process must stop in $j\le \lfloor 1/\e\rfloor$ steps. Then $[a_j,b_j]$ is the desired interval.
\end{proof}

By iterating the above lemma, we can cover most of $[a,b]$ by intervals on which $f$ is $\e$-linear. We use $\leb$ to denote Lebesgue measure.
\begin{lemma} \label{lem:tube-null-2}
For every $\e>0$ there is $\tau>0$ such that the following holds: for any $1$-Lipschitz function $f:[a,b]\to\R$ there exists a family of non-overlapping intervals $\{ [c_j,d_j]\}_{j=1}^M$ such that:
\begin{enumerate}[(\rm i)]
  \item $(f,c_j,d_j)$ is $\e$-linear for all $j$.
  \item $d_j-c_j\ge \tau(b-a)$ for all $j$.
  \item $\leb\left([a,b]\setminus \cup_j [c_j,d_j]\right) \le \e(b-a)$.
\end{enumerate}
\end{lemma}
\begin{proof}
Apply Lemma \ref{lem:tube-null-1} to the interval $[a,b]$ to obtain an interval $[x_1,y_1]\subset [a,b]$. Next, apply Lemma \ref{lem:tube-null-1} to the intervals $[a,x_1]$ and $[y_1,b]$ to obtain intervals $[x_2,y_2]$ and $[x_3,y_3]$ (we allow for degenerate intervals). Continue inductively. In each step, a proportion at least $\delta=\delta(\e)$ is removed from the length of the set $[a,b]\setminus \cup_{k=1}^{2^\ell-1} [x_k,y_k]$, so after a number $N(\e)$ of steps the remaining length is at most $(\e/2)|b-a|$. The total number of intervals is $2^{N(\e)}-1$. Hence, if $\tau = \e 2^{-(N(\e)+1)}$, the intervals of length $<\tau|b-a|$ contribute length at most $(\e/2)|b-a|$. Removing them from the collection of all $[x_k,y_k]$ we obtain the desired collection of intervals $[c_j,d_j]$.
\end{proof}

\begin{corollary} \label{cor:tube-null-3}
For every $\e>0$ there is $\tau=\tau(\e)>0$ such that the following holds: for any $1$-Lipschitz function $f:[a,b]\to \R$ there exists a family of non-overlapping intervals $\{ [c_j,d_j]\}_j$, such that:
\begin{enumerate}[(\rm i)]
  \item $(f,c_j,d_j)$ is $\e$-linear for all $j$.
  \item $\tau(b-a) \le d_j - c_j \le  c_j$ for all $j$.
  \item $\leb\left([a,b]\setminus \cup_j [c_j,d_j]\right) \le \e|b-a|$.
\end{enumerate}
\end{corollary}
\begin{proof}
Let $\tau=\tau(\e)$ be the number given by Lemma \ref{lem:tube-null-2}. Let $k_0$ be the largest (negative) integer such that $2^{k_0+1}\ge \e$. Apply Lemma \ref{lem:tube-null-2} to the intervals
\[
I_k = [a+2^k(b-a),a+2^{k+1}(b-a)], \quad k=k_0,\ldots,-2,-1,
\]
and collect all resulting intervals. It is easy to check that the conclusion holds with $2\e$ in place of $\e$ and $\e\tau$ in place of $\tau$, which is a formally equivalent statement.
\end{proof}

The next lemma is similar to Lemma \ref{lem:tube-null-2}, but we get the additional information that the slopes of $f$ on the sub-intervals are increasing; the price to pay is that $f$ becomes $\e$-superlinear on the sub-intervals (instead of $\e$-linear).
\begin{lemma} \label{lem:tube-null-4}
Given $\e>0$ there is $\tau=\tau(\e)>0$ such that the following holds. Let $f: [a,b]\to \R$ be a $1$-Lipschitz function. Then there exists a  collection of non-overlapping intervals $\{ [a_k,b_k] \}_k$ such that $b_{k+1}\le a_k$ and:
\begin{enumerate}[(\rm i)]
  \item  $(f,a_k,b_k)$ is $\e$-superlinear for all $k$.
  \item $b_k-a_k \ge \tau (b-a)$ for all $k$.
  \item $\leb\left([a,b]\setminus \cup_k [a_k,b_k]\right) \le \e|b-a|$.
  \item The sequence $s_f(a_k,b_k)$ is increasing.
\end{enumerate}
\end{lemma}
\begin{proof}
The idea (once $b_k$ is given) is to define $a_k$ as the number $a<b_k$ that maximizes $s_f(a,b_k)$. Unfortunately, for a general Lipschitz function such number may not exist and, more importantly, it can be too close to $b_k$. We deal with this issue by restricting $a$ to the set of endpoints of intervals given by  Lemma \ref{lem:tube-null-2}. We proceed to the details.

Without loss of generality, $f(a)=0$. Let $\tau = \tau(\e^2/4)$ be the number given by  Lemma \ref{lem:tube-null-2} applied with $\e^2/4$ in place of $\e$, and let $\{ [c_i,d_i]\}_i$ be the intervals given by the lemma. Write $E=[a,b]\setminus \cup_i (c_i,d_i)$, so that $\leb(E)\le \e^2(b-a)/4$. Let $\mathcal{C} = \{ c_i\}\cup \{ d_i\}\cup \{ a\}$. For each $y\in (a,b]$ let $x=P(y)$ be the element of $\mathcal{C}\cap [a,y)$ that maximizes $s_f(x,y)$ (if there are several such elements, pick the largest one). Let $y_0=b$ and so long as $y_j>a$ inductively set $y_{j+1}=P(y_j)$. By construction, the sequence $s_f(y_{j+1},y_j)$ is increasing. Let
\begin{align*}
\mathcal{J}_1 &= \{ j: |[y_{j+1},y_j] \cap E |\ge (\e/2) (y_j-y_{j+1}) \},\\
\mathcal{J}_2 &= \{ j:  y_j-y_{j+1} < \tau (b-a) \}.
\end{align*}
We let $\{ [a_k,b_k]\}$ be the collection $\{ [y_{j+1},y_j] : j\notin \mathcal{J}_1\cup\mathcal{J}_2 \}$, ordered so that $b_k \le a_{k+1}$. The claims that $s_f(a_k,b_k)$ is increasing and $b_k-a_k\ge \tau (b-a)$ are clear. Also,
since $\leb(E)\le \e^2(b-a)/4$,
\[
\sum_{j\in\mathcal{J}_1} (y_j-y_{j+1})  \le 2\leb(E)/\e \le  (\e/2)(b-a).
\]
Likewise, any interval $[y_{j+1},y_j]$ with $y_j-y_{j+1}<\tau (b-a)$ is contained in $E$, and therefore
\[
\sum_{j\in\mathcal{J}_2} (y_j-y_{j+1}) \le \leb(E) \le (\e/2)(b-a).
\]
It follows that
\[
\leb\left([a,b]\setminus \cup_k [a_k,b_k]\right)\le \e(b-a).
\]
It remains to prove that if $j\notin\mathcal{J}_1\cup \mathcal{J}_2$, then $(f,y_{j+1},y_j)$ is $\e$-superlinear. For simplicity, we write $y=y_{j+1},y'=y_j$. By the definition of $\mathcal{J}_1$, the interval $[y, y']$ can be split into the union of some intervals $([c_i, d_i])_{i\in I}$ plus a remainder set of measure at most $(\e/2)(y'-y)$.  Since $y=P(y')$, we know that
\[
s_f(x,y')\le s_f(y,y'), \quad x\in \{c_i,d_i\}, i\in I,
\]
which in turn implies that
\[
L_{f,c_i,d_i}(x) \ge L_{f,y,y'}(x), \quad x\in [c_i,d_i], i\in I.
\]
Since $f$ is $(\e^2/4)$-linear on $[c_i,d_i]$, we get that, for $x\in [c_i,d_i]$,
\[
f(x) \ge L_{f,c_i,d_i}(x) - (\e^2/4)(d_i-c_i) \ge L_{f,y,y'}(x) - (\e^2/4)(y'-y).
\]
Finally, if $x\in [y,y']\cap E$, we can find $x'=c_i$ or $x'=d_i$ with $|x'-x|\le (\e/2)(y'-y)$. Applying what we already know to $x'$ and the $1$-Lipschitz property of $f$, we conclude
\[
f(x) \ge f(x') - (\e/2)(y'-y) \ge L_{f,y,y'}(x) - (\e^2/4+\e/2)(y'-y),
\]
completing the proof.
\end{proof}

We are now able to prove the main result of this section. The statement is similar to that of the previous lemmas, but the crucial new element is that the slopes of $f:[0,1]\to\R$ on many of the sub-intervals are bounded away from $0$ and $1$. Note that this cannot hold for the function $f$ that has slope $0$ on $[0,1-s]$ and slope $1$ on $[1-s,1]$, where $s=f(1)$. In other words, we must have $f(1-s)>f(0)$. It turns out that this assumption is also sufficient, with the parameters measuring ``many'' and ``bounded away'' unsurprisingly depending on the difference $f(1-s)-f(0)$.

\begin{prop} \label{prop:superlinear-decomposition}
Given $s\in (0,1)$, $t\in (0,1)$ there is $\xi=\xi(s,t)>0$ such that the following holds for all sufficiently small $\e\le\e_1(s,t)$ and $\tau=\tau(\e)>0$.

Let $f:[0,B]\to\R$ be a  non-decreasing, $1$-Lipschitz function with
\[
f(0)=0,\quad f((1-s)B)\ge tB,\quad f(B)=sB.
\]
Then there is a collection of intervals $\{ [a_j,b_j] \}_j$ such that $b_{j+1}\le a_j$ and:
\begin{enumerate}[(\rm i)]
  \item  \label{it:lip-prop:i} $(f,a_j,b_j)$ is $\e$-superlinear for all $j$.
  \item \label{it:lip-prop:ii} $\tau B \le b_j - a_j \le a_j$ for all $j$.
  \item \label{it:lip-prop:iii}  $\leb\left([0,B]\setminus \cup_j [a_j,b_j]\right) \le \e B$.
  \item \label{it:lip-prop:iv}  $\sum \{ b_j - a_j : s_f(a_j,b_j)\in [\xi,1-\xi] \} \ge \xi B$.
\end{enumerate}
Moreover, the values of $\xi$ and $\e_1$ can be chosen to be uniform over $s,t$ varying in any compact subset of $(0,1)$.
\end{prop}
\begin{proof}
Replacing $f(x)$ by $g(x)=f(Bx)/B:[0,1]\to\R$ we may and do assume that $B=1$.

Let $\sigma=\sigma(s,t)>0$ be a small enough number to be chosen later. We split $[0,1]$ as a union of non-overlapping intervals $\{ I_n=[c_n,c_{n+1}]\}_{n=0}^N$ with $c_n$ increasing such that:
\begin{enumerate}
  \item \label{it:I0-4sigma} $|I_0|= 4\sigma$,
  \item \label{it:In-sigma} For each $n\in [1,N]$, $\sigma \le |I_n|\le 2\sigma$,
  \item There is $n_0\in [1,N]$ such that $1-s=c_{n_0}$.
\end{enumerate}
This can be easily arranged if $\sigma$ is small enough in terms of $s$. For example, we can take $I_0=[0,4\sigma]$, and then split each of $[4\sigma,1-s]$ and $[1-s,1]$ into intervals of length equal to or slightly larger than $\sigma$. We note the following consequence of \eqref{it:I0-4sigma} and \eqref{it:In-sigma} that will be used later: for each $n\in [1,N]$,
\[
|I_n|+|I_{n+1}| \le c_1 \le c_n.
\]

Fix $\e_0$; we will later require it to be small enough in terms of $\sigma$. We will eventually choose $\e=4\e_0/\sigma$. Let $\tau_0=\tau_0(\e_0)$ be the smaller of the values of $\tau$ arising from Corollary \ref{cor:tube-null-3} and Lemma \ref{lem:tube-null-4}. We apply Corollary \ref{cor:tube-null-3} to the interval $I_0$, and Lemma \ref{lem:tube-null-4} to the intervals $I_n$ for each $n\in [1,N]$, in both cases with $\e_0$ in place of $\e$, to obtain intervals $\{[a_{n,k},b_{n,k}]\}_k$. We write $s_{n,k}=s_f(a_{n,k},b_{n,k})$ for simplicity (the function $f$ is fixed throughout the proof).

We let $\zeta\in (0,\sigma)$ be a small parameter that will ultimately be chosen small enough in terms of $\sigma$ (hence in terms of $s,t$ only); in fact $\zeta=\sigma/11$ works. We subdivide the indices $n\in [1,N]$ into various disjoint classes:
\begin{align*}
\mathcal{I}_1 &= \big\{ n\in [1,N]: \sum_k \{ b_{n,k}-a_{n,k} :   s_{n,k} \le \zeta \} \ge (1-\sigma)|I_n| \big\}, \\
\mathcal{I}_2 &= \big\{ n\in [1,N]: \sum_k \{ b_{n,k}-a_{n,k} :   s_{n,k} \ge 1-\zeta \} \ge (1-\sigma)|I_n| \big\}, \\
\mathcal{I}_3 &= \big\{ n\in [1,N]\setminus (\mathcal{I}_1\cup \mathcal{I}_2): \sum_k \{ b_{n,k}-a_{n,k} :   s_{n,k} \notin [\zeta, 1-\zeta] \} \ge (1-\tfrac{\sigma}{2})|I_n| \big\}  ,\\
\mathcal{I}_4 &= [1,N]\setminus (\mathcal{I}_1\cup \mathcal{I}_2\cup \mathcal{I}_3).
\end{align*}
Roughly speaking, if $n\in\mathcal{I}_1$ then $f(c_{n+1})\approx f(c_n)$ (so that $f$ must be roughly constant on $\mathcal{I}_n$); if $n\in\mathcal{I}_2$, then $f$ is close to a linear function with slope $1$ on $I_n$; if $n\in\mathcal{I}_3$, then $f$ is close to a piecewise linear function that has slope $0$ in an initial interval $[c_n,z_n]$ and then slope $1$ on $[z_n,c_{n+1}]$ (with $z_n$ not too close to either $c_n$ or $c_{n+1}$).

Assume first that $\mathcal{I}_4\neq\varnothing$. Note that, for any $n\in [1,N]$,
\[
\sum_{k} b_{n,k}-a_{n,k} \ge (1-\e_0)|I_n| \ge (1-\tfrac{\sigma}{4})|I_n|,
\]
provided $\e_0\le \sigma/4$. If $n\in\mathcal{I}_4$, then
\[
\sum_k \{ b_{n,k}-a_{n,k} :   s_{n,k} \notin [\zeta, 1-\zeta] \} \le (1-\tfrac{\sigma}{2})|I_n|,
\]
and hence
\[
 \sum \{ b_{n,k}-a_{n,k} :   s_{n,k} \in [\zeta, 1-\zeta] \} \ge \frac{\sigma}{4}|I_n| \ge \frac{\sigma^2}{4}.
\]
On the other hand, for each $n\in [0,N]$,
\begin{align*}
b_{n,k}-a_{n,k} &\le  |I_n| \le c_n \le  a_{n,k},\\
b_{n,k}-a_{n,k} &\ge \tau_0 |I_n| \ge \sigma\tau_0,
\end{align*}
so the collection $\{ [a_{n,k},b_{n,k}]\}_{n,k}$ already works (we take $\e=\e_0$ in this case). We hence assume that $\mathcal{I}_4$ is empty from now on.

\textbf{Claim 1}. There is $n\in [n_0,N]$ such that $n\notin\mathcal{I}_2$, provided $\sigma$ is chosen small enough in terms of $s,t$; in fact, $\sigma<1-\sqrt{1-t/s}$ is enough (any value of $\sigma$ works if $t/s>1$).

Indeed, suppose $n\in\mathcal{I}_2$ for all $n\in [n_0,N]$. Since
\[
f(b_{n,k}) = f(a_{n,k}) + s_{n,k}(b_{n,k}-a_{n,k}),
\]
and $f$ is non-decreasing, it easily follows from the definition of $\mathcal{I}_2$ and the inequality $\zeta\le\sigma$ that
\[
f(c_{n+1}) \ge f(c_n) + (1-\sigma)^2 (c_{n+1}-c_n)
\]
for all $n\in\mathcal{I}_2$. Adding from $n_0$ to $N$, and using that $f(c_{n_0})=f(1-s)\ge t$, we get
\[
s= f(1) \ge t + (1-\sigma)^2 \cdot s,
\]
which contradicts the choice of $\sigma$.

\textbf{Claim 2}. There is $n\in [1,n_0]$ such that $n\notin\mathcal{I}_1$, provided $\sigma<t/6$.

To see this, assume on the contrary that $n\in \mathcal{I}_1$ for all $n\in [0,n_0]$ and note that since $f$ is $1$-Lipschitz and $\zeta\le\sigma$,
\begin{align*}
f(c_1) &\le |I_0|= 4\sigma,\\
f(c_{n+1}) &\le f(c_n) + 2\sigma|I_n| \quad\text{for all }n\in [1,n_0],
\end{align*}
and hence, telescoping,
\[
f(1-s) = f(c_{n_0}) \le f(c_1) + 2\sigma \le 6\sigma.
\]
This contradicts the assumption $f(1-s)>t$ if $\sigma<t/6$.

In conclusion, since we are assuming that $\mathcal{I}_4=\varnothing$, there must exist $n\in [2,N]$ such that $n-1\in\mathcal{I}_2\cup \mathcal{I}_3$ and $n\in \mathcal{I}_1\cup\mathcal{I}_3$. We work with this fixed value of $n$ for the rest of the proof. Unpacking the definitions, this implies that
\begin{align*}
\sum \{ b_{n-1,k}-a_{n-1,k} :   s_{n-1,k} \ge 1-\zeta \} &\ge \frac{\sigma}{2}|I_{n-1}|,\\
\sum \{ b_{n,k}-a_{n,k} :   s_{n,k} \le \zeta \} &\ge \frac{\sigma}{2}|I_n|.
\end{align*}
Let
\[
k'=\min \{ k: s_{n-1,k} \ge 1-\zeta\}.
\]
Write $\wt{a} = a_{n-1,k'}$, and note that $\wt{a}\le c_n-\tfrac{\sigma}{2}|I_{n-1}|$. Recall that $s_{n-1,k}$ is increasing in $k$, $f$ is $\e_0$-superlinear on $[a_{n-1,k},b_{n-1,k}]$,  and  $\sum_k b_{n-1,k}-a_{n-1,k} \ge (1-\e_0)|I_{n-1}|$. It then follows that
\begin{equation} \label{eq:lower-bound-a-tilda-interval}
f(x) \ge f(\wt{a}) +  (1-\zeta)(x-\wt{a}) -2\e_0|I_{n-1}|\quad\text{for all }x\in [\wt{a},c_n].
\end{equation}
Likewise, if we set
\[
k'' = \max\{ k: s_{n,k} \le \zeta \},
\]
and write $\wt{b}=b_{n,k''}$, then $\wt{b}\ge c_n+\tfrac{\sigma}{2}|I_n|$ and
\begin{equation} \label{eq:lower-bound-b-tilda-interval}
f(c_n) \le f(x) \le f(\wt{b}) \le f(x) + \zeta(\wt{b}-x) +\e_0 |I_n| \quad\text{for all }x\in [c_n,\wt{b}].
\end{equation}
We will show that
\begin{equation} \label{eq:slope-xi}
\zeta \le s_f(\wt{a},\wt{b}) \le 1- \zeta.
\end{equation}
provided  $\zeta\le\sigma/11$ and $\e_0\le\zeta$. Indeed:
\begin{align*}
s_f(\wt{a},\wt{b})  &\le \frac{ (c_n-\wt{a}) + (\zeta+\e_0) |I_n|}{\wt{b}-\wt{a}} & (\text{\eqref{eq:lower-bound-b-tilda-interval} for }x=c_n) \\
&\le 1-\frac{|I_n|(\sigma/2-2\zeta)}{\wt{b}-\wt{a}}& (\wt{b}\ge c_n+\tfrac{\sigma}{2}|I_n|)\\
&\le 1 - \frac{1}{3}(\sigma/2-2\zeta)\le  1- \zeta& (\eqref{it:In-sigma},\zeta\le \sigma/10).
\end{align*}
On the other hand,
\begin{align*}
s_f(\wt{a},\wt{b}) &\ge \frac{f(c_n)-f(\wt{a})}{\wt{b}-\wt{a}}& (f \text{ non-decreasing})\\
&\ge \frac{(c_n-\wt{a})(1-\zeta) -2\e_0|I_{n-1}|}{\wt{b}-\wt{a}}& (\eqref{eq:lower-bound-a-tilda-interval} \text{ for } x=c_n)\\
&\ge \frac{(\sigma/2)|I_{n-1}|(1-\zeta) -2\e_0|I_{n-1}|}{ |I_{n-1}|+|I_n| }& ( \wt{a}\le c_n-\tfrac{\sigma}{2}|I_{n-1}|)\\
&\ge \frac{1}{3}(\sigma/2(1-\zeta)-2\zeta)\ge\zeta&  ( \eqref{it:In-sigma}, \zeta\le \sigma/11, \sigma\le 1).
\end{align*}
We have verified that \eqref{eq:slope-xi} holds. We will now use this to show that $(f,\wt{a},\wt{b})$ is $\e$-superlinear, where we define $\e=4 \e_0/\sigma$. It follows from \eqref{eq:lower-bound-a-tilda-interval} and the right-hand side inequality in \eqref{eq:slope-xi} that, for $x$ in the interval $[\wt{a},c_n]$,
\[
f(x) \ge L_{f,\wt{a},\wt{b}}(x) - 2\e_0|I_{n-1}|.
\]
Now, for $x\in [c_n,\wt{b}]$ we know from \eqref{eq:lower-bound-b-tilda-interval} that
\begin{align*}
f(x) &\ge f(\wt{b}) - \zeta(\wt{b}-x) -  \e_0|I_n| \\
&= f(\wt{a}) + s_f(\wt{a},\wt{b})(\wt{b}-\wt{a}) -\zeta(\wt{b}-x) -  \e_0|I_n| \\
&\ge f(\wt{a}) + s_f(\wt{a},\wt{b})(x-\wt{a}) -   \e_0|I_n| = L_{f,\wt{a},\wt{b}}(x) - \e_0|I_n|,
\end{align*}
using that $s_f(\wt{a},\wt{b})\ge \zeta$ in the last line. Since $\max(|I_{n-1}|,|I_n|) \le (\sigma/2)^{-1}(\wt{b}-\wt{a})$, we have shown that $(f,\wt{a},\wt{b})$ is $(4\e_0/\sigma)$-superlinear, as claimed.

Now the collection consisting $[\wt{a},\wt{b}]$ together with all the intervals $\{[a_{\wt{n},k},b_{\wt{n},k}]\}_{\wt{n}=0}^N$ disjoint from $[\wt{a},\wt{b}]$ is easily seen to satisfy the claims in the proposition: the intervals $[a_{\wt{n},k},b_{\wt{n},k}]$ satisfy \eqref{it:lip-prop:i}-\eqref{it:lip-prop:ii} by the same reasoning as in the case $\mathcal{I}_4\neq\varnothing$, while we have verified that $(f,\wt{a},\wt{b})$ is $\e$-superlinear,
\[
\sigma^2/2 \le \sigma|I_{n-1}|/2 \le |\wt{b}-\wt{a}| \le |I_{n-1}|+|I_n| \le c_{n-1} \le \wt{a},
\]
and \eqref{it:lip-prop:iv} follows with $\xi=\min(\sigma^2/2,\zeta)$ from $|\wt{b}-\wt{a}|\ge \sigma^2/2$ and \eqref{eq:slope-xi}. Finally, \eqref{it:lip-prop:iii} also follows since the total length of the intervals is at least the sum of the lengths of $[a_{\wt{n},k},b_{\wt{n},k}]$, which is at least $1-\e_0\ge 1-\e$ by construction.

Finally, it is clear that all the parameters appearing in the proof can be taken to work in a neighborhood of $s$ and $t$, and this completes the proof of the proposition.
\end{proof}

\subsection{The proof of Theorem \ref{thm:multiscale-decomposition}, and a variant}

We now convert the combinatorial decompositions from the previous section into suitable multiscale decompositions of regular measures, in particular proving Theorem \ref{thm:multiscale-decomposition}. We start with a variant of Proposition \ref{prop:superlinear-decomposition} in which the endpoints of the intervals $[a_j,b_j]$ lie on a lattice.
\begin{corollary} \label{cor:Lipschitz-decomp-grid}
Let $f:[0,B]\to \R$ be a $1$-Lipschitz function satisfying the assumptions of  Proposition \ref{prop:superlinear-decomposition}. Then in Proposition \ref{prop:superlinear-decomposition} we may further request that the numbers $a_j$ and $b_j$  lie in $(B/\ell)\N_0$ provided $\ell\ge \ell_0(\e)$ is sufficiently large (after changing the values of $\tau$ and $\xi$ slightly).
\end{corollary}
\begin{proof}
We may assume $B=1$. Let $\{ [a_j,b_j]\}$ be the intervals given by Proposition \ref{prop:superlinear-decomposition}, and set
\[
\wt{a}_j = \frac{\lceil \ell a_j \rceil}{\ell},\quad \wt{b}_j = \frac{\lfloor \ell a_j \rfloor}{\ell}.
\]
Then $a_j \le \wt{a}_j \le a_j+1/\ell$ and $b_j-1/\ell \le \wt{b}_j \le b_j$. Since $f$ is $1$-Lipschitz and non-decreasing, the same inequalities are preserved when applying $f$. If $\ell$ is taken large enough in terms of $\tau$ (hence in terms of $\e$), then $\wt{b}_j-\wt{a}_j \ge b_j -a_j - \tau^2$ (say) so that \eqref{it:lip-prop:ii} and \eqref{it:lip-prop:iii} continue to hold with $\tau/2$ in place of $\tau$. Likewise, a short calculation shows that if $\ell$ is large enough then
\[
|s_{f,\wt{a}_j,\wt{b}_j}-s_{f,a_j,b_j}| \le \e,
\]
and if $(f,a_j,b_j)$ is $\e$-superlinear, then $(f,\wt{a}_j,\wt{b}_j)$ is $(2\e)$-superlinear. Using these facts it is easy to check that \eqref{it:lip-prop:i} and \eqref{it:lip-prop:iv} continue to hold for the intervals $[\wt{a}_j,\wt{b}_j]$, with $2\e$ in place of $\e$ and, say, $\xi/2$ in place of $\xi$.
\end{proof}

We can now complete the proof of Theorem \ref{thm:multiscale-decomposition}.
\begin{proof}[Proof of Theorem \ref{thm:multiscale-decomposition}]
Let
\begin{align*}
\xi&=\inf_{s\in [u/4,1-u/2]}\xi(s,u/(2d))>0,\\
\e_1&=\inf_{s\in [u/4,1-u/2]}\e_1(s,u/(2d))>0,
\end{align*}
be the values given by Proposition \ref{prop:superlinear-decomposition}, and note that they depend only on $u$ and $d$. Fix $\e\in (0,\e_1)$, and let $\tau(\e), \ell_0(\e)$ be the numbers given by Proposition \ref{prop:superlinear-decomposition}. Take $\ell\ge \ell_0(\e)$.

Let $f:[0,\ell]\to [0,\ell]$ be the function such that $f(0)=0$,
\[
f(j) = \frac{1}{d} (\sigma_1+\ldots+\sigma_j) \quad (j=1,\ldots,\ell),
\]
and interpolates linearly between $j$ and $(j+1)$. Since $\sigma_i\in [0,d]$, the function $f$ is $1$-Lipschitz and non-decreasing. Let $\{ [A_j,B_j]\}_{j=1}^M$ be the intervals provided by Proposition \ref{prop:superlinear-decomposition}. In light of Corollary \ref{cor:Lipschitz-decomp-grid}, we may and do assume that $A_j, B_j\in\N_0$.

Let
\[
s= \frac{f(\ell)}{\ell}=\frac{\sigma_1+\ldots+\sigma_{\ell}}{\ell}.
\]
Since $|X|\ge 2^{-dm}$, it follows from \eqref{eq:single-scale-Frostman} that $\mu(Q)\lesssim 2^{-um}$ for all $Q\in\cD_m$, and therefore $\cN(X,m)\gtrsim 2^{um}$. On the other hand, we may cover $[0,1]^d$ by $\lesssim |X|^{-1}$ balls of radius $|X|^{1/d}$, whence $1\lesssim |X|^{-1} 2^{-um}$. Hence, if $T$ is large enough in terms of $u$, we may assume
\begin{equation} \label{eq:intermediate-dimension}
2^{(u/2-d)m} \le |X| \le 2^{-um/2}.
\end{equation}
Now, from this and Lemma \ref{lem:frostman-regular}\eqref{it:frostman-regular-ii}, we get that, provided $1/T\le \e\le u/2$,
\[
1-s-\e \le \frac{\log |X|}{-md} \le 1-s \Longrightarrow s\in [u/4,1-u/2].
\]
Let $j= \lfloor (1-s)\ell\rfloor$, and note that any cube $Q$ of side length $2^{-j T}$ can be covered by $C_{d,T}=O_{d,T}(1)$ balls of radius $2^{-(1-s)m}\le |X|^{1/d}$. It follows from our assumption \eqref{eq:single-scale-Frostman} that
\[
\log\mu(Q)\le  \log(C_{d,T})-um \le (\e-u)m,
\]
provided $\ell$ is large enough in terms of $\e,d,T$. On the other hand, we get from Lemma \ref{lem:frostman-regular}\eqref{it:frostman-regular-i} that for all $2^{-j T}$-dyadic cubes $Q$ hitting $X$, we have the lower bound
\[
\log\mu(Q) \ge -\e m -\left((m/\ell)(\sigma_1 +\ldots +\sigma_{\lfloor (1-s)\ell\rfloor}\right),
\]
provided $T$ is large enough in terms of $\e$. Comparing the upper and lower bounds on $\log\mu(Q)$, we conclude
\[
f((1-s)\ell) \ge \frac{1}{d}(\sigma_1 +\ldots +\sigma_{\lfloor (1-s)\ell\rfloor}) \ge \frac{1}{d}(u-2\e)\ell \ge \frac{u\ell}{2d},
\]
provided $\e$ is small enough in terms of $u$. We have therefore checked that the hypotheses of Proposition \ref{prop:superlinear-decomposition} are satisfied with our choice of parameters.

Claim \eqref{it:multiscale-i} is clear. Let $\alpha_j = d s_{f}(A_j,B_j)$. By the definition of $f$,
\[
s_{f}(A_j,B_j) = \frac{\sigma_{A_{j+1}}+\cdots+\sigma_{B_j}}{d(B_j-A_j)} .
\]
Since $(f,A_j,B_j)$ is $\e$-superlinear,
\[
d(f(A_j+k)-f(A_j))=\sigma_{A_{j+1}}+\ldots+\sigma_{A_j+k} \ge k\alpha_j - d\e(B_j-A_j),
\]
for any $k\in [0,B_j-A_j]$. Fix $Q\in \cD_{T A_j}$. It follows from the first part of Lemma \ref{lem:frostman-regular} that
\[
\mu^Q(R) \le 2^{d\e T(B_j-A_j)} 2^{-k T \alpha_j}
\]
for all $R\in \cD_{kT}(\supp_{\textsf{d}}(\mu^Q))$. Any ball $B=B(x,r)$ with $x\in \supp(\mu^Q)$, $r\in [2^{-m_j},1]$ can be covered by $O_T(1)$ $2^T$-adic cubes of side length $2^{-kT}\le r$. Therefore, provided $\ell$ is taken large enough in terms of $\e, T$,
\[
\mu^Q(B(x,r)) \le 2^{2 d\e m_j} r^{\alpha_j}
\]
for all $r\in [2^{-m_j},1]$ and all $x\in\supp(\mu^Q)$. This gives \eqref{it:multiscale-ii}, with $2d\e$ in place of $\e$.

For the third claim, write $[0,\ell]\setminus \cup_j [A_j,B_j] = \cup_i [C_i,D_i]$ with the intervals $[C_i,D_i]$ non-overlapping. We know from Corollary \ref{cor:tube-null-3} that $\sum_i D_i-C_i \le \e\ell$, and $s_f(C_i,D_i)\le 1$ simply because $f$ is $1$-Lipschitz. On the other hand, telescoping,
\[
\ell  s_f(0,\ell) = \sum_j (B_j-A_j) s_f(A_j,B_j) + \sum_i (D_i-C_i) s_f(C_i,D_i).
\]
Recalling that $\alpha_j = d s_{f}(A_j,B_j)$, we deduce that
\[
\sum_j \alpha_j m_j \ge  \left(\tfrac{1}{\ell}(\sigma_1+\ldots+\sigma_\ell)-\e\right)m.
\]
Applying Lemma \ref{lem:frostman-regular}\eqref{it:frostman-regular-i} and taking $T\ge 1/\epsilon$, we obtain \eqref{it:multiscale-iii}.

Finally, from the last part of Proposition \ref{prop:superlinear-decomposition} we obtain
\[
  \sum\{  m_j: \alpha_j \in [\xi,d-\xi]\} \ge T \sum\{ B_j-A_j: s_f(A_j,B_j)\in [\xi/d,1-\xi/d]\}\ge \xi m.
\]
Hence \eqref{it:multiscale-iv} holds, and this completes the proof.
\end{proof}

In a very similar (but simpler) way, we obtain the following:
\begin{theorem} \label{thm:multiscale-decomposition-Ahlfors}
For every $\e>0$ there is $\tau=\tau(\e)>0$ such that the following holds for all sufficiently large $T\ge T_0(\e)$ and for all large enough $\ell\ge \ell_0(T,\e)$:

Let $\mu$ be a $((\sigma_1,\ldots,\sigma_\ell);T)$-regular measure on $[0,1)^d$ with support $X$, and write $m=\ell T$.  Then there are a collection of pairwise disjoint intervals $\{[A_j,B_{j+1})\}$ contained in $[0,\ell)$ and numbers $\alpha_j$, such that the following hold:
\begin{enumerate}[(\rm i)]
  \item \label{it:multiscale-i-Frostman} $ \tau \ell \le B_j - A_j \le A_j$ for all $j$.
  \item \label{it:multiscale-ii-Frostman}  Write $m_j= T(B_j-A_j)$. For each $Q\in \cD_{T A_j}$,
\[
 2^{-\e m_j} r^{\alpha_j} \le  \mu^Q(B(x,r)) \le 2^{\e m_j} r^{\alpha_j}\quad\text{for all } r\in [2^{-m_j},1], x\in \supp(\mu^Q).
\]
  \item \label{it:multiscale-iii-Frostman}
  \[
  \sum_j \alpha_j m_j \ge (\tfrac{\sigma_1+\ldots+\sigma_\ell}{\ell}-\e)m.
  \]
  \end{enumerate}
\end{theorem}
The proof is the same as that of Theorem \ref{thm:multiscale-decomposition}, except that we rely on Corollary \ref{cor:tube-null-3} (providing a decomposition into $\e$-linear, rather than superlinear, pieces) instead of Proposition \ref{prop:superlinear-decomposition}.

\section{Proof of main results}
\label{sec:proofs}

\subsection{A multiscale lower bound for entropy}

In this section we prove Theorems \ref{thm:main-continuous} and \ref{thm:main-discrete}. The proofs rely on a multiscale bound for the entropy of (nonlinear) projections, that goes back in various forms to \cite{Hochman14, Orponen17, Shmerkin19, KeletiShmerkin19}. We denote the orthogonal projection onto a plane $V\in\mathbb{G}(d,k)$ by $P_V$.

\begin{prop} \label{prop:good-bound-from-below}
Fix $1\le k<d$. Let $\mu,\nu\in\cP([0,1)^d)$ and let $[A_i,B_i)_{i=1}^q$ be disjoint sub-intervals of $(0,m]$ such that $B_i\le 2 A_i$. Let $F:U\to \R^k$ be a $C^2$ map defined in a neighborhood of $\supp(\mu)$, such that $DF(x)$ has full rank $k$ for all $x\in\supp(\mu)$. Denote
\[
V(x) = \ker(DF(x))^{\perp} \in \mathbb{G}(d,k).
\]
Assume furthermore that $\nu\le\Delta\mu$. Then
\[
 H_m(F\nu) \ge - O_{F,d}(q)   + \sum_{i=0}^{q-1} \sum_{Q\in\cD_{A_i}} \nu(Q)  H_{B_i-A_i}^{m \Delta}\left(P_{V(x_Q)}\mu^Q\right),
\]
where  $x_Q$ is an arbitrary point in $Q$. The constant implicit in $O_{F,d}(q)$ depends in a continuous manner on the $C^2$ norm of $F$.
\end{prop}
Since this particular statement does not appear in the literature, but the proof is a minor modification of arguments in \cite{Hochman14, Orponen17, KeletiShmerkin19}, we defer the proof to Appendix \ref{sec:entropy-of-projections}.

Note that, in the case $k=1$, we have $V(x)$ is the line generated by $\dir F'(x)$, and $P_{V(x)}(\cdot)=\langle \dir F'(x),\cdot\rangle$. This is the only case needed for the proofs of Theorems  \ref{thm:main-continuous} and \ref{thm:main-discrete}.

A key feature of Proposition \ref{prop:good-bound-from-below} is that even though the map $F$ is (possibly) nonlinear, the lower bound on $H_m(F\nu)$ involves only linear projections. This is why we'll ultimately be able to deduce statements about families of smooth maps from the (linear) projection theorems discussed in Section \ref{sec:projection-theorems}.

\subsection{Proof of Theorem \ref{thm:main-discrete}}

The number $\kappa$ and the family $\mathcal{F}=\{F_\lambda:\lambda\in\Lambda\}$ are given. The proof will involve a number of other small parameters, whose dependencies are as follows: $\e_1=\e_1(\kappa)$, $\xi=\xi(\kappa)$,  $\eta=\eta(\kappa,\xi)$, $\eta'=\eta'(\eta,\xi)$, $\delta_0=\delta_0(\kappa)$,  $\zeta=\zeta(\kappa,\xi,\eta,\e_1,\delta_0)$, $\tau=\tau(\zeta)$. All the parameters are also allowed to depend on the ambient dimension $d$.

We will show that the claim holds with $\eta'/2$ in place of $\eta$, provided $\e=\e(\kappa,\zeta,\tau,\eta')$ is taken small enough, and $m=T\ell$, where $T=T(\e,\zeta)$, $\ell=\ell(\zeta, T,\mathcal{F})$ are sufficiently large integers. If $m$ is not of the form $T\ell$, we consider instead $m'=T\lfloor m/T\rfloor$.

Let
\[
\mu=\mathbf{1}_X/|X|,
\]
and let $Y$ be the $2^{-m}$-set given by Lemma  \ref{lem:subset-regular} applied to $\mu$. Taking $T$ large enough in terms of $\e$, we may assume
\begin{equation} \label{eq:Y-large-X}
|Y| \ge 2^{-(\e/2) m}|X|.
\end{equation}
The non-concentration hypothesis \eqref{eq:hyp-non-concentration} yields
\[
\mu_Y(B(x,|Y|^{1/d})) \le \mu(Y)^{-1} \mu(B(x,|X|^{1/d})) \le 2^{(\e/2-\kappa) m} \le 2^{-\kappa m/2},
\]
taking $\e$ small enough in terms of  $\kappa$.  We can therefore apply Theorem \ref{thm:multiscale-decomposition} to the measure $\mu_Y$, with $u=\kappa/2$. Let $\e_1=\e_1(\kappa/2)$ and $\xi=\xi(\kappa/2)$ be the numbers given by the theorem. Pick $0<\zeta<\e_1$, let $\tau=\tau(\zeta)$, and suppose that $T$ is chosen large enough in terms of $\zeta$, and $\ell$ is chosen large enough in terms of $\zeta$ and $T$, that the conclusions of Theorem \ref{thm:multiscale-decomposition} hold (with $\zeta$ in place of $\e$).

Let $\{[A_j,B_j)\}_j$ and $\{\alpha_j\}_j$ be the intervals and exponents obtained from Theorem \ref{thm:multiscale-decomposition} applied to $\mu_Y$. Write
\[
m_j=T (B_j-A_j).
\]
Let $\eta=\eta(\kappa), \delta_0=\delta_0(\kappa)$ be the numbers provided by Theorem \ref{thm:projection-combined}. Our aim is to apply Theorem \ref{thm:projection-combined} to the measures $\mu_Y^Q$, $Q\in \cD_{T A_j}$ and $\rho_x:= \theta_x\nu_{\Lambda_x}$, and the scales $2^{-m_j}$. Since $m_j\ge \tau m$ by the first part of Theorem \ref{thm:multiscale-decomposition}, the non-concentration assumption \eqref{eq:hyp-circle-non-concentration} on $\rho_x$ implies that
\begin{equation} \label{eq:rho-x-Frostman}
\rho_x(H^{(r)}) \le 2^{(\e/\tau) m_j}  r^\kappa \quad (H\in \mathbb{G}(d,d-1), r\in [2^{-m_j},1]).
\end{equation}
Given $Q\in \cD_{T A_j}$, we know from Theorem \ref{thm:multiscale-decomposition}\eqref{it:multiscale-ii}  that
\[
\mu_Y^Q(B(x,r)) \le 2^{\zeta m_j} r^{\alpha_j} \quad(r\in [2^{-m_j},1]).
\]
Thus, if $\zeta$ and then $\e$ are chosen so that $\e/\tau\le \zeta<\delta_0(\kappa)$, the hypotheses of Theorem \ref{thm:projection-combined} are met with $\zeta$ in place of $\delta$. Note that by making $\ell$ large enough in terms of all other parameters, we can make $\tau m_j$ large so that the theorem is applicable.

Recall that given $x\in [0,1)^d$,  the only element of $\cD_{T A_j}$ containing $x$ is denoted by $Q_{T A_j}(x)$. We apply Theorem \ref{thm:projection-combined} to the measures $\mu_Y^{Q_{T A_j}(x)}$ and $\rho_x$, to obtain sets $E(x,j)$ with
\[
\rho_x(E(x,j))\le 2^{-\zeta m_j} \le 2^{-\zeta \tau m},
\]
such that
\[
e\notin E(x,j) \Longrightarrow P_e\mu_Y^{Q_{T A_j}(x)} \text{ is } (\gamma(\alpha_j),\zeta,m_j)\text{-robust},
\]
where $\gamma$ is the function from Theorem \ref{thm:projection-combined}. Note that $\gamma(a)\ge a/d-6\zeta$ for all $a\in [0,d]$. A little algebra shows that, provided $\eta<\kappa/(2d)$ (which we may assume) and $\zeta$ is small enough in terms of $\kappa$, then
\[
\gamma(a) \ge  a/d+\eta  \text{ if }  a\in [(\eta+6\zeta)d/(d-1),d-d\eta-6d\zeta].
\]
Making $\eta$ and $\zeta$ smaller in terms of $\xi$ and $d$ only, we may further assume that
\begin{equation} \label{eq:estimate-gamma}
\gamma(a) \ge \left\{
\begin{array}{ccc}
  a/d-6\zeta & \text{ if } & a\notin [\xi,d-\xi] \\
  a/d+\eta & \text{ if } & a\in [\xi,d-\xi]
\end{array}
\right..
\end{equation}

Let $E(x)=\cup_j E(x,j)$. Then
\begin{equation}  \label{eq:rho-bad-set-small-measure}
\rho_x(E(x)) \le (\ell/\tau) 2^{-\zeta \tau m} \le 1/2,
\end{equation}
assuming $\ell$ is large enough in terms of $\zeta$. Recalling that $\rho_x = \theta_x\nu_{\Lambda_x}$ and that $\nu(\Lambda_x)\ge 2^{-\e m}$ by assumption, this shows that there exists a set $G_x\subset\Lambda$ such that $\nu(G_x)\ge 2^{-\e m}/2$ and $\theta_x(\lambda)\notin E(x)$ for all $\lambda \in G_x$. A standard argument shows that the set $\{ (\lambda,x):\theta_x(\lambda)\notin E(x)\}$ is Borel. By Fubini, we can find $\lambda\in \supp(\nu)$ and a set $Z\subset Y$ with
\begin{equation} \label{eq:Z-large-Y}
|Z|\ge \tfrac{1}{2} 2^{-\e m}|Y|\ge 2^{-2\e m}|X|,
\end{equation}
and such that $\lambda \in G_x$ for all $x\in Z$. We work with this value of $\lambda$ for the rest of the proof, and note that
\begin{equation}  \label{eq:robust-conclusion}
P_{\theta_x(\lambda)}\mu_Y^{Q_{T A_j}(x)} \text{ is } (\gamma(\alpha_j),\zeta,m_j)\text{-robust for all $j$ and $x\in Z$}.
\end{equation}
On the other hand, we estimate
\begin{align*}
\sum_j \gamma(\alpha_j). m_j &\ge \sum_{j:\alpha_j\notin [\xi,d-\xi]} (\alpha_j/d- 6\zeta)m_j + \sum_{j:\alpha_j\in [\xi,d-\xi]} (\alpha_j/d+\eta)m_j & (\text{Eq. } \eqref{eq:estimate-gamma})\\
&\ge -6\zeta m + \frac{1}{d}\sum_{j} \alpha_j m_j + \eta\sum_{j:\alpha_j\in [\xi,d-\xi]}  m_j\\
&\ge -6\zeta m + (\log\cN(Y,m) - 2\zeta)/d + \eta\xi m & (\text{Thm }\ref{thm:multiscale-decomposition}\eqref{it:multiscale-iii}\text{-}\eqref{it:multiscale-iv})\\
&\ge \log \cN(X,m)/d +(\eta\xi -9\zeta) m & (\text{Eq. } \eqref{eq:Y-large-X},\e\le \zeta)
\end{align*}
Hence, making $\zeta$ small enough in terms of $\xi$ and $\eta$, and writing $\eta'=\xi\eta/2$,
\begin{equation} \label{eq:multiscale-dim-gain}
\frac{1}{m}\sum_j \gamma(\alpha_j). m_j  \ge \frac{\log\cN(X,m)}{d m} + \eta'.
\end{equation}

Note that the number $q$ of intervals $[A_j,B_j]$ is at most $1/\tau$.  For each $Q\in \cD_{T A_j}$ such that $Q\cap Z\neq\varnothing$, pick some $x_Q\in Q\cap Z$. We are now ready to apply Proposition \ref{prop:good-bound-from-below}. Suppose $Z'\subset Z$, $|Z'|\ge 2^{-\e m}|Z|$. Recalling  \eqref{eq:Z-large-Y}, we get that $|Z'|\ge 2^{-2\e m-1}|Y|$. Appyling Proposition \ref{prop:good-bound-from-below} to $\nu=\mu_{Z'}$, $\mu=\mu_Y$, $\Delta=2^{2\e m+1}$, and the sequence $([T A_j, T B_j])_{j=1}^q$, we get:
\begin{equation} \label{eq:lower-bound-multiscale-entropy}
 H_m(F_\lambda\mu_{Z'}) \ge - O_{\mathcal{F}}(1/\tau)   + \sum_{j=1}^{q} \sum_{Q\in\cD_{T A_j}: Q\cap Z'\neq\varnothing} \mu_{Z'}(Q)  H_{m_j}^{m 2^{2\e m+1}}\left(P_{\theta_{x_Q}(\lambda)}\mu_Y^Q\right).
\end{equation}
Recall from \eqref{eq:robust-conclusion} that $P_{\theta_{x_Q}(\lambda)}\mu_Y^Q$ is $(\gamma(\alpha_j),\zeta,m_j)$-robust. Lemma \ref{lem:robust-to-entropy} then yields that (provided $\ell$ is large enough in terms of $\zeta$, $\e$ and $\tau$, which makes $m_j$ large enough in terms of $\zeta$, $\e$)
\[
H^{2^{\zeta m_j/2}}_{m_j}(P_{\theta_{x_Q}(\lambda)}\mu_Y^Q) \ge  (\gamma(\alpha_j)-\e)m_j.
\]
If $\e$ is small enough in terms of $\zeta$ of $\tau$, then
\[
2^{\zeta m_j/2} \ge 2^{\zeta \tau m/2} \ge m 2^{2\e m+1}.
\]
Recalling \eqref{eq:multiscale-dim-gain} and \eqref{eq:lower-bound-multiscale-entropy}, we conclude that, provided $m$ is large enough in terms of $\tau$, $\e$ and $\mathcal{F}$,
\begin{align*}
H_m(F_\lambda\mu_{Z'}) &\ge -\e m + \sum_{j=1}^{q} \sum_{Q\in\cD_{T A_j}: Q\cap Z'\neq\varnothing} \mu_{Z'}(Q) (\gamma(\alpha_j)-\e)m_j\\
&\ge \log\cN(X,m)/d + (\eta'-2\e)m.
\end{align*}
Thus, assuming $\e<\eta'/4$,
\[
\log\cN(F_\lambda(Z'),m) \ge  \log\cN(X,m)/d +\eta'/2,
\]
giving the claim with $\eta'/2$ in place of $\eta$ and $Z$ in place of $X'$.

\subsection{\texorpdfstring{The case $\Lambda_x=\Lambda$ for all $x$}{The case of Lambda constant}}

We indicate what changes are needed in the proof to obtain the stronger conclusions when $\Lambda_x=\Lambda$. As before, let $\mu=\mathbf{1}_X/|X|$. We need to apply Lemma \ref{lem:decomposition-regular} to $\mu$ (with the parameter $1.5\e$ in place of $\e$) instead of Lemma \ref{lem:subset-regular}, to obtain the sets $(X_i)_i$. Taking $T$ large enough in terms of $\e$, we have $\mu(X_i)\ge 2^{-2\e m}$. Other than the $X_i$ having slightly smaller measure, depending on $\e$, the analysis we did for the set $Y$ in the proof of Theorem \ref{thm:main-discrete} carries over to each of the $X_i$ verbatim. In particular, the bound \eqref{eq:rho-bad-set-small-measure} holds for any $X_i$ and $x\in X_i$.  We take $\e$ small enough in terms of $\zeta,\tau$ so that
\[
\rho_x(E(x)) =\nu\{\lambda: \theta_x(\lambda)\in E(x) \} \le 2^{-5\e m}.
\]
Write $\wt{X}=\cup_i X_i$. By Fubini, denoting Lebesgue measure by $\cL$,
\[
(\nu\times \cL)\{(\lambda,x)\in \Lambda\times \wt{X}: \theta_x(\lambda)\in E(x) \}\le 2^{-5\e m}|\wt{X}|.
\]
It follows that the set
\[
\Lambda' = \{\lambda\in\Lambda: |x\in\wt{X}: \theta_x(\lambda)\in E(x)|\ge 2^{-4\e m}|\wt{X}|\}
\]
has $\nu$ measure at most $2^{-\e m}$.

Fix $\lambda\in\Lambda\setminus\Lambda'$ and $X'\subset X$ with $|X'|\ge 2^{-\e m}|X|$. By the first part of Lemma \ref{lem:decomposition-regular} (which we are applying with $1.5\e$ in place of $\e$), we can find $i$ with
\[
\frac{|X'\cap X_i|}{|X_i|} \ge \frac{|X'|}{|X|}-2^{-1.5\e m} \ge 2^{-\e m}/2.
\]
Since $\lambda\notin\Lambda'$,
\[
|x\in\wt{X}: \theta_x(\lambda)\in E(x)| < 2^{-4\e m}|X| < \tfrac{1}{2}|X'\cap X_i|.
\]
Let $Z= \{ x\in X'\cap X_i: \theta_x(\lambda)\notin E(x)\}\subset X'$. We have seen that $\mu_{X_i}(Z)\ge 2^{-\e m}/4$. Now starting with \eqref{eq:robust-conclusion}, the exact same argument from the previous section (with $X_i$ in place of $Y$) yields the desired conclusion
\[
\cN(F_\lambda(X'),m) \ge \cN(F_\lambda(Z),m) \ge  2^{m\eta'/2}\cN(X,m)^{1/d}.
\]

\subsection{Proof of Theorem \ref{thm:main-continuous}}

Our next goal is to prove Theorem \ref{thm:main-continuous}. We will in fact establish the following stronger fact.
\begin{theorem} \label{thm:main-continuous-measure}
Given $\kappa>0, 0<\alpha<d$ there is $\eta=\eta_d(\kappa,\alpha)>0$ (that can be taken continuous in $\kappa,\alpha$) such that the following holds.

Let $(\Lambda,\nu)$ be a Borel probability space, let $U\subset\R^d$ be an open domain, and let $\mu$ be a Borel probability measure on $U$ such that
\[
\mu(B(x,r)) \le C\, r^\alpha, \quad (x\in U, r>0).
\]
Let $F:\Lambda \times U\to\R$ be a Borel function such that, for each $\lambda\in\Lambda$, the map $F_\lambda(x)=F(\lambda,x):U\to\R$ is $C^2$ without singular points. For each $x\in \R^d$  define the map
\[
\theta_x(\lambda) = \dir(\nabla F_\lambda(x)) \in S^{d-1}.
\]
Suppose further that for $\mu$-almost all $x$ there are a set $\Lambda_x$ with $\nu(\Lambda_x)\ge c$ and a number $C_x>0$ satisfying the decay condition \eqref{eq:main-thm-cont-projective-decay}.

Then there are a constant $c'=c'(c)>0$ and  a set $\Lambda_0$ with $\nu(\Lambda_0)\ge c'$, such that the following holds: for all $\lambda\in\Lambda_0$ there is a set $A_\lambda$ with $\mu(A_\lambda)\ge c'$ such that, for all sufficiently large $m$, the measure $F_\lambda (\mu_{A_\lambda})$ is $(\alpha/d+\eta,\eta,m)$-robust.

Moreover, $c'\to 1$ as $c\to 1$.
\end{theorem}

Note that Theorem \ref{thm:main-continuous} follows by considering a Frostman measure $\mu$ on $A$ of exponent arbitrarily close to $\alpha$, and applying  Lemma \ref{lem:robust-to-dimension}. The proof of Theorem \ref{thm:main-continuous-measure} is very similar to that of Theorem \ref{thm:main-discrete}. We indicate the required changes.

By passing to a subset of nearly full $\mu$-measure, we may assume that $C_x\le C$ for all $x$. Denote $\rho_x=\theta_x\nu_{\Lambda_x}$. We have the following (stronger) analog of \eqref{eq:rho-x-Frostman}:
\begin{equation} \label{eq:positive-dim-of-directions}
\rho_x(H^{(r)}) \le C r^\kappa\quad (r\in (0,1], H\in \mathbb{G}(d,d-1)).
\end{equation}

Unlike the proof of Theorem \ref{thm:main-discrete}, we need to consider all small scales at once.  We can define the parameters $\zeta, \tau, \xi, \eta,\delta_0$ arising from Theorems \ref{thm:multiscale-decomposition} and \ref{thm:projection-combined}; these are independent of the scale. Fix a small parameter $\e$, a large integer $T$ and an even larger integer $\ell$. We continue to write $m=T\ell$. As before, $\e, T,\ell$ can depend on all the previous parameters, $T$ can depend on $\e$ and $\ell$ can depend on $T$ and $\e$.

For each large $\ell$, we apply Lemma \ref{lem:decomposition-regular} to $\mu, T$ and $\e$, to obtain a family of sets $(X_{\ell,i})_i$. For each $i$ and $x\in \cup_i X_{\ell,i}$, we define the set $E(x)$ as in the proof of Theorem \ref{thm:main-discrete}, but now denote it by $E(x,\ell)$. For completeness, we define $E(x,\ell)=\varnothing$ if $x\notin \cup_i X_{\ell,i}$. As before, for each $x\in X$ we have
\[
\rho_x(E(x,\ell)) \le (m/\tau) 2^{-\zeta \tau m}.
\]
Hence, provided $\ell_1$ is large enough in terms of $\zeta,T,\e$ only, we can ensure $\rho_x(E(x))\le \e$ for all $x\in X$, where
\[
E(x) =\bigcup_{\ell=\ell_1}^\infty E(x,\ell).
\]
Unwrapping the definitions, this means that
\[
\nu\{ \lambda\in\Lambda_x: \theta_x(\lambda)\notin E(x) \}\ge \nu(\Lambda_x)(1-\e)\ge c(1-\e).
\]
By Fubini, if we let $c'=1-\sqrt{1-c(1-\e)}$, there is a set $\Lambda_0$ with $\nu(\Lambda_0)\ge c'$ such that if $\lambda\in \Lambda_0$ then
\begin{equation} \label{eq:def-Z}
\mu(Z_\lambda)\ge c', \quad\text{where }Z_\lambda=\{ x: \theta_x(\lambda)\notin E(x)\} .
\end{equation}
We fix $\lambda\in\Lambda_0$ from now on and write $Z=Z_\lambda$.

Let $\e=\e(\kappa,\alpha)>0$, $\eta'=\eta'(\kappa,\alpha)>0$ be numbers to be determined. Suppose that $F_\lambda\mu_Z(A)\ge 2^{-\e m/2}$; our goal is to show that, provided $m$ is large enough,
\begin{equation} \label{eq:large-counting-number-to-prove}
\log\cN(A,m)\ge m(\alpha/d+\eta').
\end{equation}
It is enough to consider the case $m=\ell T$ with $\ell\ge \ell_1$ sufficiently large.

Let $Y=F_\lambda^{-1}A$. Since $\mu(Y)\ge 2^{-\e m/2}$, the first part of Lemma \ref{lem:decomposition-regular} guarantees that there is a set $X_{\ell,i}$ with $\nu(Y) \ge 2^{-\e m}$, where $\nu=\mu_{X_{\ell,i}}$. At this point we are working at a fixed (small) scale, with a regular measure $\nu$ and a set $Y$ of sufficiently large measure, which is a subset of $Z$ given in \eqref{eq:def-Z}. Then essentially the same argument from the proof of Theorem \ref{thm:main-discrete} shows that
\[
H_m(F_\lambda \nu_Y)\ge m(\alpha/d+\eta'),
\]
which, since $F_\lambda\nu_Y$ is supported on $A$, gives \eqref{eq:large-counting-number-to-prove}. The only difference is that $\nu$ might be very uniform, so that the upper bound in \eqref{eq:intermediate-dimension} needs not hold. However, this case is even better for us. Formally, let $\ol{\alpha} = (1+\alpha)/2$. Then we consider the cases $\sigma_1+\ldots+\sigma_{\ell}\le \ol{\alpha}\ell$ and $\sigma_1+\ldots+\sigma_\ell> \ol{\alpha}\ell$ separately. In the first case,  \eqref{eq:intermediate-dimension} does hold, with $u=u(\alpha)$, so we can proceed exactly as in the proof of Theorem \ref{thm:main-discrete}. Otherwise, we apply Theorem \ref{thm:multiscale-decomposition-Ahlfors} in place of Theorem \ref{thm:multiscale-decomposition}, and use the fact that the function $\gamma$ from Theorem \ref{thm:projection-combined} satisfies $\gamma(\alpha)\ge \alpha/d-7\delta$. In the end, the lower estimate we get on $H_m(\nu_Y)$ in this case is $m(\overline{\alpha}-\text{error term})$, where we can make the error term arbitrarily small, completing the proof.

\section{Applications and generalizations}
\label{sec:applications}

\subsection{Real analytic families}
\label{subsec:real-analytic}

In this section we derive several applications and generalizations of Theorems \ref{thm:main-continuous} and \ref{thm:main-discrete}. We start with a statement about one-parameter real analytic families.

\begin{theorem} \label{thm:real-analytic}
Given $d\ge 2$, $\alpha\in (0,d)$, $\kappa\in (d-2,2)$ there is $\eta=\eta_d(\alpha,\kappa)>0$ (depending continuously on the parameters) such that the following holds.

Let $F_y(x)=F(x,y):\R^d\times \R\to \R$ be a real-analytic function. Let $X\subset\R^d$ be a Borel set of Hausdorff dimension $>\alpha$, and suppose $y\mapsto \dir F'_y(x)$ is non-constant for all $x\in X$ outside of a set of dimension $<\hdim(X)$. Then
\[
\hdim\left\{y: \hdim(F_y X)\le \frac{\alpha}{d}+\eta \right\}\le \kappa.
\]
\end{theorem}
\begin{proof}
Let $\eta(\kappa,\alpha,d)$ be the number provided by Theorem \ref{thm:main-continuous}. Let $\mu$ be a Frostman measure of exponent $\alpha$ supported on the $x\in X$ for which $y\mapsto \theta_x(y):=\dir F'_y(x)$ is nonconstant. It is enough to show that if $Y$ is a compact set with $\hdim(Y)>\kappa$ then there is $y\in Y$ such that $\hdim(F_y X)>\alpha/d+\eta$.

Let $\nu$ be a $\kappa$-Frostman measure on $Y$. By analyticity and the fact that $\theta_x(y)$ is non-constant, we can find open balls $U\subset\R^d, V\subset \R$ meeting $\supp(\mu)$ and $\supp(\nu)$ respectively, and an index $j\in \{1,\ldots,d-1\}$ such that
\[
\big|\frac{d}{dy}\theta_{x,j}(y)\big| \ge c>0\quad (x\in U, y\in V),
\]
where $\theta_{x,j}$ is the $j$-th coordinate of $\theta_x$. (To be more precise, this holds in a local chart of $S^{d-1}$.) Then $\theta_x\nu_V$ satisfies a Frostman condition of exponent $\kappa$, uniformly in $x\in U$. Indeed, it is enough to prove this for the projection $\theta_{x,j}\nu_V$ of $\theta_x\nu_V$ to its $j$-th coordinate, but this holds since $\theta_{x,j}$ is  bi-Lipschitz on $V$, uniformly in $x\in U$.

Since $\kappa>d-2$, the assumptions of Theorem \ref{thm:main-continuous} are satisfied with $\Lambda=U\cap \supp(\nu)$  and $V\cap \supp(\mu)$ in place of $X$ (and $\Lambda_x\equiv\Lambda$). The conclusion of  Theorem \ref{thm:main-continuous} is exactly what we were trying to prove.
\end{proof}

\begin{remark}
The proof shows that, instead of analyticity, it is enough to assume that $F$ is $C^2$ and the set of zeros of $\frac{d}{dy}\theta_x(y)$ has Hausdorff dimension $<\kappa$ for all $x$ outside of a set of Hausdorff dimension $<\hdim(X)$.
\end{remark}

\subsection{Distance sets}
\label{subsec:distance-sets}

In this section we apply Theorem \ref{thm:main-continuous} to prove  Theorem \ref{thm:main-distance-sets}. In order to verify that the assumptions of Theorem \ref{thm:main-continuous}, we rely on the following spherical projection estimate. Let
\[
e(x,y) = \frac{y-x}{|y-x|}
\]
be the direction spanned by two different points in $\R^d$, and write $e_x(y)=e(x,y)$ for the spherical projection with center $x$.
\begin{theorem} \label{thm:radial-positive-dim}
Fix $1\le k<d$. Given $\kappa,\alpha>0$ there are $c,\eta>0$ (depending continuously on $\alpha,\kappa$) such that the following holds.

Let $\mu,\nu\in\cP(B^d(0,1))$ be measures satisfying decay conditions
\begin{align*}
\mu(V^{(r)}) &\le C_\mu r^{\kappa},\\
\nu(V^{(r)}) &\le C_\nu r^{\alpha},
\end{align*}
for all $V\in \mathbb{A}(d,k-1)$ and $0<r\le 1$. (If $k=1$, $V^{(r)}$ is a ball of radius $r$.) Suppose $\nu$ gives zero mass to every affine $k$-plane.
Then  for all $x$ in a set $E$ of $\mu$-measure $\ge 1/2$ there is a set $K=K(x)$ with $\nu(K(x))\ge 1/2$ such that
\begin{equation} \label{eq:spherical-proj-non-concentration}
e_x(\nu_{K(x)})(W^{(r)}) \le r^\eta \quad r\in (0,r_0], W\in \mathbb{G}(d,k),
\end{equation}
where $r_0>0$ depends only on $d, \mu, C_\nu, \alpha$.

Finally, the set $\{ (x,y) : x\in E, y\in K(x) \}$ is compact.
\end{theorem}
In dimension $d=2$, this is due to T.~Orponen \cite{Orponen19}. The argument in higher dimensions is similar; for completeness, we include the details of the proof in Appendix \ref{sec:spherical-projections}.

\begin{proof}[Proof of Theorem \ref{thm:main-distance-sets}]
We first prove the theorem for the Euclidean metric, and comment on the changes required to handle other norms at the end.

We may assume that $\hdim(X)>\alpha$, provided we can show the value of $\eta$ is continuous in the parameters (we can then apply the claim to $\alpha-\e$ in place of $\alpha$ for a suitably small $\e$ to deduce that the statement also holds for $\hdim(X)=\alpha$). It is enough to show that there is $\eta=\eta_d(\alpha,\kappa)>0$ such that, for all compact sets $Y$ with $\hdim(Y)> \kappa$, there is $y\in Y$ satisfying $\hdim(\Delta^y X)\ge \alpha/d+\eta$. We may assume that $Y$ is disjoint from $X$ and both $X$ and $Y$ are contained in the unit ball. Let $\mu$ be an $\alpha$-Frostman measure on $X$.

Let $\nu$ be a Frostman measure on $Y$ with exponent $\kappa$. Our goal is to apply Theorem \ref{thm:main-continuous} with $\Lambda=Y$ and $F_y(x)=|x-y|$. Note that
\[
F'_y(x)=y-x/|y-x|=e(x,y)\in S^{d-1}.
\]
Thus the maps $\theta_x$ from Theorem \ref{thm:main-continuous} are the spherical projections $e_x(y)=e(x,y)$.

Suppose first that there is a hyperplane $P$ with $\nu(P)>0$; we can then assume that $Y=P$. We consider two sub-cases. If $\hdim(X\cap P)\ge \alpha$, then we are in the $(d-1)$-dimensional setting and we can argue by induction to get an even better estimate (note that the base case $d=1$ is trivial). Otherwise, $\hdim(X\setminus P)=\hdim(X)$, so we may assume that $X\cap P=\varnothing$.

Now, for $x\notin P$, the map $y\mapsto e_x(y):P\cap B(0,1) \to S^{d-1}$ is bi-Lipschitz onto its image, with a constant that depends on $\dist(x,P)$. Hence $e_x\nu(B_r)\lesssim_x r^\kappa$ for all $r$-balls $B_r$ and $x\in X$. For every linear hyperplane $H\in \mathbb{G}(d,d-1)$ we can cover $S^{d-1}\cap H^{(r)}$ by $\lesssim r^{-(d-2)}$ balls of radius $r$. Since $\kappa>d-2$, this shows that the assumptions of
Theorem \ref{thm:main-continuous} hold (with $\kappa-(d-2)$ in place of $\kappa$). Theorem \ref{thm:main-continuous} then gives the desired conclusion.

Suppose now that $\nu$ gives zero mass to all hyperplanes. In this case we apply Theorem \ref{thm:radial-positive-dim} with $k=d-1$. As before, the hypotheses hold because $\kappa,\alpha>d-2$. Let $\kappa'=\kappa'(\alpha,\kappa)>0$ be the value given by Theorem \ref{thm:radial-positive-dim}. The theorem provides with a set $X'$ with $\mu(X')\ge 1/2$, and sets $K(x)$ for each $x\in X'$ with $\nu(K(x))\ge 1/2$, such that \eqref{eq:spherical-proj-non-concentration} holds. The claim now follows from Theorem \ref{thm:main-continuous} applied to the set $X'$ (with parameters $\alpha$ and $\kappa'$).

We now consider the case of a $C^2$ norm $N$ with unit ball of everywhere positive Gaussian curvature. Note that $N'(x)=N'(x/|x|)$ for all non-zero $x$. Hence, defining $\varphi(e)=N'(e):S^{d-1}\to \R^d$ and $\psi=\dir\circ\,\varphi:S^{d-1}\to S^{d-1}$, we have
\[
\theta_x(y)=\dir \frac{d}{d x} N(x-y) = \dir\varphi(e_x(y))=\psi(e_x(y)),
\]
whence $\theta_x=\psi\circ e_x$. Now the hypothesis of everywhere positive Gaussian curvature translates to the map $\psi(e)$ having non-vanishing Jacobian (this can be seen e.g. by expressing the unit sphere of $N$ as a graph $y=f(x)$ in local coordinates, and noting that the Jacobian of $\psi$ is essentially the Hessian of $f$). Therefore we can cover $S^{d-1}$ by finitely many patches on which $\psi$ is a diffeomorphism (onto its image) and hence bi-Lipschitz. It follows that any decay condition enjoyed by $e_x\nu$ is still valid for $\psi e_x\nu=\theta_x\nu$, with different constants but the same exponent. With this observation, the proof in the Euclidean case goes through unchanged.
\end{proof}

\begin{remark} \label{rem:general-norms}
The hypotheses on the norm $N$ can be substantially weakened. What the argument above actually uses is that the Gauss map of the unit ball of $N$ locally has a Lipschitz inverse. This can be relaxed further; for example, if the Gauss map locally has a H\"{o}lder inverse, then the argument still goes through, except that the dimension gain $c$ will depend also on the H\"{o}lder exponent of these local inverses. In particular, Theorem \ref{thm:main-distance-sets} extends, for example, to all $\ell^p$ norms for $p\in (1,\infty)$.
\end{remark}

\subsection{Discretized incidences}
\label{subsec:incidences}

As noted in the introduction, the circle of problems studied in this paper are related to incidence counting in the discretized setting. In this section we show how Theorem \ref{thm:main-discrete} can be used to deduce non-trivial incidence counting bounds, under suitable assumptions. Let $\{ C_a:a\in B^2(0,1)\}$ be a parametrized family of planar curves. Fix $\delta$-separated sets $E, A\subset B^2(0,1)$. We are interested in bounding the size of the (discretized) incidence set
\[
\mathcal{I}(E,A) = \{ (p,a)\in E\times A: \dist(p,C_a)<\delta \}.
\]
Let us make the further assumptions that $E\subset X\times \R$ for some $\delta$-separated set $X\subset [-1,1]$, and that each curve meets vertical lines in a uniformly bounded number of points. In this case, we have the trivial bound
\[
|\mathcal{I}(E,A)|  \lesssim |X||A|.
\]
In general, this bound is sharp, for example if $E^{(2\delta)}$ is a small square and all curves parametrized by $A$ cross this square from side to side. It is also sharp in cases where both $E$ and $A$ satisfy strong non-concentration assumptions. The examples are given by ``train tracks'', a well known object that many problems in discretized geometry have to grapple with (see, for example, \cite[Figure 1]{KatzTao01}). For concreteness, let $X\subset[-1,1]$ be a well separated set of size $\approx \delta^{-1/2}$ (for example, it could be an arithmetic progression of gap $\delta^{-1/2}$), and let $E$ be a maximal $\delta$-separated subset of $X\times [0,\delta^{1/2}]$ (so $E^{(2\delta)}$ resembles a train track). Then if $A$ is any set of $\delta$-separated lines that cross the rectangle $[0,1]\times [0,\delta^{1/2}]$ from side to side, the incidence count $|\mathcal{I}(E,A)|$ equals $|X||A|$. The family of all such lines is essentially (parametrized by) a ball of radius $\delta^{1/2}$, and is therefore maximally concentrated, but if we take a spread-out subset, such as horizontal lines $y=\delta j$ with $j\in [0,\delta^{-1/2}]$, then $A$ is highly non-concentrated, and we still have $|\mathcal{I}(E,A)|\approx |X||A|$. The next theorem shows that, under non-concentration assumptions on $X$ and $A$, and a further hypothesis that rules out train-track examples, an improvement over the trivial bound can be achieved.

\begin{theorem} \label{thm:incidences}
Given $\kappa\in (0,1)$ and $c>0$ there is $\e(\kappa)>0$ such that the following holds for all $\delta$ small enough in terms of $\kappa, c$.

Let $G:B^2(0,1)\times [-1,1]\to\R$ be a $C^2$ map of $C^2$ norm $\le c^{-1}$, and such that
\[
|\partial G/\partial a(\cdot)| \ge c \quad\text{on } B^2(0,1)\times [-1,1].
\]
Assume, further, that
\begin{equation} \label{eq:non-triviality-assumption}
\left|\frac{d}{dx} \dir \left( \frac{\partial G}{\partial a}(\cdot)\right)\right| \ge c\quad\text{on } B^2(0,1)\times [-1,1].
\end{equation}

Let $E,A\subset B^2(0,1)$, $X\subset [-1,1]$ be $\delta$-separated sets such that $E\subset X\times \R$. Further, assume that
\begin{align}
|E|&\le \delta^{-\e}|X||A|^{1/2}, \label{eq:size-assumption-E} \\
|A\cap B(a,\delta|A|^{1/2})|&\le \delta^{\kappa} |A| &(a\in B^2(0,1)),  \nonumber\\
|X\cap B(x,r)| &\le \delta^{-\e} r^{\kappa}|X|& (x\in [-1,1], r\in [\delta,1]). \nonumber
\end{align}
Then
\[
| \{(p,a)\in E\times A: \dist(p,\textup{Graph}(G_a))<\delta\} | \le \delta^\e |X||A|.
\]
\end{theorem}

Before presenting the deduction of this theorem from Theorem \ref{thm:main-discrete},  we make some remarks and then give some examples.

\begin{remark}
The assumption \eqref{eq:size-assumption-E} rules out train track examples. Consider again lines $y=ax+b$ parametrized by $(a,b)$. Let $X\subset [-1,1]$ be a well separated set with $|X|=\delta^{-t}$, let $E=X\times \{j\delta\}_{j=0}^{\delta^{-t}}$, and let $A_0=\{j\delta\}_{j=0}^{\delta^{-t}}\times \{j\delta\}_{j=0}^{\delta^{-t}}$. Then $|E|\approx |X||A_0|^{1/2}$ and $|\cI(E,A_0)|\approx |X||A_0|$. Of course, in this case $A_0$ is maximally concentrated. But if we allowed $|E|$ to be much larger than $|X||A|^{1/2}$, we could take $A$ to be large non-concentrated subset of $A_0$. Then all the hypotheses of the theorem would hold, except for \eqref{eq:size-assumption-E}, and the conclusion would clearly fail. There is nothing special about lines here - for more general graphs, one has to consider ``curved train tracks'' of the form $C^{(\delta^{s})}\cap (X\times\R)$ for a fixed graph $C$ in the family.
\end{remark}


\begin{remark}
In the case of lines, Theorem \ref{thm:incidences} follows from Bourgain's discretized projection theorem, except that our non-concentration requirement on the lines is weaker (and in some sense the weakest possible one). We note that (for lines), M.~Bateman and V.~Lie \cite{BatemanLie19} have previously obtained a related result, but under different assumptions. Roughly speaking, they assume more on the family of lines, and on the structure of the set $E$, but on the other hand, they have a single-scale non-concentration hypothesis on $X$ (while we have a single-scale non-concentration hypothesis on the set $A$ of lines). Recently, L.~Guth, N.~Solomon and H.~Wang \cite{GSW19} studied incidences between tubes satisfying the strongest possible non-concentration assumption, and obtained sharp incidence bounds in that regime.
\end{remark}

\begin{remark}
We have tried to give an indication of the relationship between our results and incidence counting, but Theorem \ref{thm:incidences} is not the strongest or most general formulation possible. The hypothesis \eqref{eq:non-triviality-assumption} can be weakened, although of course some non-degeneracy condition on the family of curves is needed. Concretely, the number on the left-hand side of \eqref{eq:non-triviality-assumption} can be allowed to be zero on a set small enough that, if we remove a suitable neighborhood of it, the incidence counting does  not change. Using the results in \S\ref{subsec:higher-rank} below, one can also deduce versions in higher dimensions, and for higher-dimensional families of curves, although the required assumptions become more restrictive and cumbersome to verify.
\end{remark}

We now give some examples of families of curves to which Theorem \ref{thm:incidences} applies. The verifications are straightforward calculations.
\begin{corollary}
Let $\kappa,\e, E, A$ and $X$ be as in Theorem \ref{thm:incidences}. For each of the following families of curves, provided $\delta$ is small enough in terms of $\kappa$ and the family, we have
\[
|\mathcal{I}(E,A)| \le \delta^\e |X||A|.
\]
\begin{enumerate}[(i)]
\item The family of lines $y=ax+b$, with $(a,b)$ in some fixed bounded set.
\item The family of parabolas $y=ax^2+bx$, with $(a,b)$ in some fixed bounded set.
\item The family of circles with center $(a,0)$ and radius $r$, with $(a,r)$ in some bounded set.
\item The family of circles of unit radius and centre in a bounded set $S$, provided the projection of $S$ onto the $x$-axis is separated from $(X-1)\cup (X+1)$.
\end{enumerate}
\end{corollary}

\begin{proof}[Proof of Theorem \ref{thm:incidences}]
Suppose $\delta=2^{-m}$. Assume on the contrary that
\[
| \{(p,a)\in E\times A: d(p,\textup{Graph}(G_a))<2^{-m}\} | > 2^{-\e m} |X||A|.
\]
Let $E_x=\{ y: (x,y)\in E\}$. Then
\[
\sum_{x\in X} |\{a\in A: d(G_a(x),E_x)<2^{-m} \} | > 2^{-\e m}|X||A|.
\]
Hence, if we set $A_x=\{ a\in A:  d(G_a(x),E_x)<2^{-m}\}$, we have
\[
|\{ x\in X: |A_x|\ge \tfrac{1}{2} 2^{-\e m}|A|\} | \ge \tfrac{1}{2} 2^{-\e m}|X|.
\]
Let $X_0\subset X^{(2^{-m})}$ be the $2^{-m}$-neighborhood of the set appearing on the left-hand side, and write $\nu=\mathbf{1}_{X_0}/|X_0|$. Using $|X_0|\gtrsim 2^{-\e m}|X^{(2^{-m})}|$ and the non-concentration assumption on $X$, we see that
\[
\nu(B(x,r))\lesssim 2^{2\e m} r^\kappa, \quad r\in [2^{-m},1].
\]
Our goal is to apply Theorem \ref{thm:main-discrete} to the family $\{ a\mapsto G_a(x)\}$ and the set $A^{(2^{-m})}$. Note that $X$ plays the role of $\Lambda$.  Using \eqref{eq:non-triviality-assumption}, the same argument in the proof of Theorem \ref{thm:real-analytic} shows that
\[
\theta_a \nu(B(y,r)) \lesssim_c 2^{2\e m} r^\kappa,
\]
uniformly over $a\in A$. Using the assumptions, we see that if $\e$ is small enough in terms of $\kappa$, then the hypotheses of Theorem \ref{thm:main-discrete} are met (with $A^{(2^{-m})}$ in place of $X$). By definition of $X_0$, we have $|A_x|\gtrsim 2^{-\e m}|A|$ for $x\in X_0$ so, again assuming $\e$ is small enough only in terms of $\kappa$, Theorem \ref{thm:main-discrete} (applied to the $2^{-m}$-neighborhoods of $A, A_x$) guarantees that, provided $m$ is large enough in terms of $c$ and $\kappa$,
\[
|E_x| \gtrsim \cN(\{ G_a(x):a\in A_x\},m) \gtrsim   2^{\eta m}|A|^{1/2} \quad\text{for all }x\in X_0,
\]
where $\eta=\eta(\kappa)>0$ is the value provided by Theorem \ref{thm:main-discrete}. We conclude
\[
|E| \gtrsim 2^{\eta m}|X_0||A|^{1/2} \gtrsim 2^{(\eta-\e)m} |X||A|^{1/2}.
\]
This contradicts \eqref{eq:size-assumption-E} if $\e$ is small enough compared to $\eta$ (hence in terms of $\kappa$ only), finishing the proof.
\end{proof}

\subsection{A higher rank non-linear projection theorem}
\label{subsec:higher-rank}

W. He \cite[Theorem 1]{He20} extended Bourgain's discretized projection theorem to higher rank projections. Using his result, we can obtain a corresponding higher rank version of Theorems \ref{thm:main-continuous} and \ref{thm:main-discrete}. The proofs are nearly the same; we only need to take a little care in extending Theorem \ref{thm:projection-combined} properly.

Following \cite{He20}, given $V\in \mathbb{G}(d,k)$ and $W\in \mathbb{G}(d,d-k)$, we define $d(V,W)=\det(P_{W^{\perp}|V})$, where $P_{W^{\perp}|V}$ denotes the restriction of $P_{W^{\perp}}$ to $V$. See \cite[Eq. (13)]{He20} for this characterization of $d(V,W)$. We remark that $d$ is \emph{not} a metric; to begin with, $V$ and $W$ live in different spaces. Also, $d(V,W)=0$ if and only if $\dim(V\cap W)\ge 1$. However, $d$ is symmetric. For $W\in \mathbb{G}(d,d-k)$, we will denote
\[
\cV(W,r) = \{ V\in \mathbb{G}(d,k): d(V,W)<r\}.
\]
Given $W\in \mathbb{G}(d,d-k)$, the set $\{ V\in \mathbb{G}(d,k): d(V,W)=0\}$ is a smooth hypersurface in $\mathbb{G}(d,k)$. Hence, $\cV(W,r)$ can be seen as a neighborhood of this hypersurface.

Later we will need to deal with the case $k=d-1$, which is particularly simple. If $H\in \mathbb{G}(d,d-1)$, $\ell\in \mathbb{G}(d,1)$, then
\begin{equation} \label{eq:dist-hyperplane-line}
d(H,\ell) = |\det(P_{H^{\perp}|\ell})| = |\cos(\angle(H^{\perp},\ell))|.
\end{equation}

The following lemma will help us achieve the correct generalization of Theorem \ref{thm:projection-combined}.
\begin{lemma} \label{lem:linear-algebra}
Fix $1\le k\le d-1$.
\begin{enumerate}
\item[\textup{(i)}]
Given $x\in S^{d-1}\cap W$, with $W\in \mathbb{G}(d,d-k)$,
\[
d(W,V)\le \dist(x,V).
\]
\item[\textup{(ii)}]
Given $x\in S^{d-1}\cap W^\perp$ with $W\in \mathbb{G}(d,d-k)$,
\[
d(W,V)\le |P_V(x)|.
\]
\end{enumerate}
\end{lemma}
\begin{proof}
\begin{enumerate}
\item[\textup{(i)}] Note that $\dist(x,V)=|P_{V^\perp}(x)|$. Hence if we extend $x$ to an orthonormal basis of $W$, pick an orthonormal basis for $V^\perp$, and write the matrix of $P_{V^{\perp}|W}$ in these bases, the first row has norm $\dist(x,V)$, while all the rows have norm at most $1$. Hence
    \[
    d(W,V)=|\det(P_{V^{\perp}|W})|\le \dist(x,V).
    \]
\item[\textup{(ii)}] Since $|P_V(x)|=\dist(x,V^{\perp})$, the claim follows from the first one and the identity $d(W,V)=d(V^\perp,W^\perp)$, see \cite[Eq. (14)]{He20}.
\end{enumerate}
\end{proof}

\begin{corollary} \label{cor:rho-decay-comparison}
Fix $1\le k\le d-1$. Let $\rho\in\cP(\mathbb{G}(d,k))$ be a measure satisfying the decay condition
\begin{equation} \label{eq:rho-decay}
\rho(\cV(W,r)) \le C\, r^\kappa \quad\text{for all }0<r\le 1, W\in \mathbb{G}(d,d-k).
\end{equation}
Then for every $x\in \R^d\setminus\{0\}$,
\begin{align}
\rho\{V:\dist(x,V)\le r\} &\le C(r/|x|)^{\kappa}, \label{eq:rho-decay-1}\\
\rho\{V: |P_V(x)|\le r\} &\le C(r/|x|)^{\kappa}. \label{eq:rho-decay-2}
\end{align}
\end{corollary}
\begin{proof}
By rescaling, we may assume $x$ has unit norm. Choosing $W$ such that $x\in W$ (resp. $x\in W^{\perp}$), the claim is immediate from Lemma \ref{lem:linear-algebra}.
\end{proof}

The non-concentration condition \eqref{eq:rho-decay} is the one appearing in He's projection theorem, \cite[Theorem 1]{He20}. On the other hand, \eqref{eq:rho-decay-1} is the required decay in order to get the higher rank version of Theorem \ref{thm:projection-Kaufman} and \eqref{eq:rho-decay-2} is the decay needed to obtain the higher rank version of Theorem \ref{thm:projection-Falconer}. The proofs are very similar to the rank $1$ case, see \cite[\S 5.3]{Mattila15} for details. Note that even though in \cite[\S 5.3]{Mattila15} there is a Frostman condition (on balls) for $\rho$, the exponents are chosen large enough so that \eqref{eq:rho-decay-1}--\eqref{eq:rho-decay-2} hold; see in particular \cite[Eqs. (5.11) and (5.12)]{Mattila15}. Thus Corollary \ref{cor:rho-decay-comparison} shows that \eqref{eq:rho-decay} is enough to obtain the decay required in all three projection regimes. As a consequence, the same proof of Theorem \ref{thm:projection-combined} yields:

\begin{theorem} \label{thm:projection-combined-high-rank}
Fix $1\le k\le d-1$. Given $0<\kappa<1$ there exists $\eta=\eta(\kappa)>0$ such that the following holds for all sufficiently small $\delta\le \delta_0(\kappa)$ and all sufficiently large $m\ge m_0(\delta)$. Fix $\alpha\in [0,d]$.  Let $\mu\in\cP_m^d$, and let $\rho\in\cP(\mathbb{G}(d,k))$ satisfy
\begin{align*}
\mu(B(x,r)) &\le 2^{\delta m} r^\alpha \quad\text{for all } r\in [2^{-m},1], x\in [0,1)^d,\\
\rho(\cV(W,r)) &\le 2^{\delta m} r^\kappa \quad\text{for all } r\in [2^{-m},1], W\in \mathbb{G}(d,d-k).
\end{align*}

Then there is a set $E\subset \mathbb{G}(d,k)$ with $\rho(E)\le 2^{-\delta m}$ such that $P_V\mu$ is $(\gamma(\alpha),\delta,m)$-robust for all $V\in \mathbb{G}(d,k)\setminus E$, where
\[
\gamma(\alpha) = \left\{
\begin{array}{ccc}
  \alpha-6\delta & \text{ if } & \alpha < \kappa/2 \\
  \frac{k}{d}\alpha+\eta & \text{ if } & \kappa/2 \le \alpha \le d-\kappa/2 \\
  k-6\delta & \text{ if } &  \alpha > d-\kappa/2
\end{array}
\right..
\]
\end{theorem}

In turn, the proofs of Theorems \ref{thm:main-continuous} and \ref{thm:main-discrete} carry over to the higher rank setting with only notation changes. Of course the decay assumption on $\rho$ must be of the form \eqref{eq:rho-decay}, and the dimension/exponent of the ``good'' projections becomes $\tfrac{k}{d}\alpha+\eta$. For example, the following is the higher rank analog of Theorem \ref{thm:main-continuous}.

\begin{theorem} \label{thm:main-continuous-higher-rank}
Fix $1\le k\le d-1$. Given $\kappa>0, 0<\alpha<d$ there is $\eta=\eta_{d,k}(\kappa,\alpha)>0$ (that can be taken continuous in $\kappa,\alpha$) such that the following holds.

Let $(\Lambda,\nu)$ be a Borel probability space and let $U\subset\R^d$ be an open domain. Let $F:\Lambda\times U\to \R^k$ be a Borel map such that $F_\lambda(x)=F(\lambda,x)$ is $C^2$ for all $\lambda\in\Lambda$, and $DF_\lambda(x)$ has full rank $k$ for all $(\lambda,x)\in \Lambda\times U$. For each $x\in U$  define the map
\[
V_x(\lambda) =  \ker(DF_\lambda(x))^{\perp} \in \mathbb{G}(d,k).
\]
Let $A\subset U$ be a Borel set of Hausdorff dimension $\ge \alpha$ such that for all $x\in A$ there are a set $\Lambda_x$ with $\nu(\Lambda_x)>0$ and a number $C_x>0$ satisfying
\begin{equation} \label{eq:main-thm-cont-projective-decay-higher-rank}
V_x\nu_{\Lambda_x}\left(\cV(W,r)\right)   \le C_x \, r^\kappa \quad\text{for all }W\in \mathbb{G}(d,d-k), r\in (0,1].
\end{equation}
Then there is $\lambda\in\supp(\nu)$ such that
\[
\hdim(F_\lambda A) \ge \frac{k}{d}\,\alpha+\eta.
\]
\end{theorem}

\subsection{Dimension of spherical projections}
\label{subsec:spherical}

It is well known that if $A\subset\R^d$ is a Borel set of Hausdorff dimension $>d-1$, then the set of directions spanned by $A$ has full measure in $S^{d-1}$ (this follows e.g. from the Marstrand-Mattila Intersection Theorem). This clearly fails if $\hdim(A)=d-1$, as $A$ can then be contained in a hyperplane. But what if $A$ is not contained in a hyperplane? One might conjecture that then the set of directions spanned by $A$ has full dimension $d-1$, and maybe even there is $x\in A$ such that the spherical projection $e_x(A)$ has full dimension (Recall that $e_x(y)=x-y/|x-y|$). This is wide open in all dimensions $d\ge 2$, but Orponen \cite[Theorem 1.5]{Orponen19} proved the following partial result.

\begin{theorem} \label{thm:orponen-radial}
Let $A,E\subset\R^2$ be Borel sets such that $\hdim(E)>0$ and $E$ is not contained in a line. Then there is $y\in E$ such that
\[
\hdim(e_y(A\setminus\{y\})) \ge \frac{\hdim(A)}{2}.
\]
In particular, if $A\subset\R^2$ is not contained in a line, it spans a set of directions of dimension $\ge \hdim(A)/2$.
\end{theorem}

Very recently, B.~Liu and C-Y.~Shen \cite[Theorem 1.3]{LiuShen20} combined Orponen's approach with Bourgain discretized sum-product theorem to improve this, by showing that if $A\subset\R^2$ has Hausdorff dimension $\alpha\in (0,2)$ and is not contained in a line, then it spans a set of directions of dimension $\ge \alpha/2+c(\alpha)$, where $c(\alpha)>0$. They also obtained a ``pinned'' version, see \cite[Theorem 1.4]{LiuShen20}.  See also \cite[Theorem 1.6]{BLZ16} for an earlier, closely related ``single scale'' result.  As a corollary of our framework, we obtain a pinned version that is slightly different from that of Liu and Shen, and partially extend it  to higher dimensions. We note that Orponen's proof of Theorem \ref{thm:orponen-radial} features heavily in our argument via Theorem \ref{thm:orponen-radial}(see Appendix \ref{sec:spherical-projections} for more details).

\begin{theorem} \label{thm:radial-projections-dim}
Fix $d\ge 2$. Given $\alpha,\kappa\in (d-2,d)$, there is $\eta=\eta_d(\alpha,\kappa)>0$ such that the following holds.

Let $A,E\subset\R^d$ be Borel sets such that $\hdim(A)=\alpha$, $\hdim(E)=\kappa$ and $E$ is not contained in a hyperplane. Then there is $y\in E$ such that
\[
\hdim(e_y A) \ge \frac{(d-1)\alpha}{d} + \eta.
\]
\end{theorem}

\begin{remark}
Compared to \cite[Theorem 1.4]{LiuShen20} (which only considers the planar case), the case $d=2$ of the theorem has the natural hypothesis that $E$, rather than $A$ as in \cite{LiuShen20}, is not contained in a line; this allows us to avoid the dichotomy in \cite[Theorem 1.4]{LiuShen20}. On the other hand, unlike in \cite[Theorem 1.4]{LiuShen20}, the gain $\eta$ depends on the dimensions of both $A$ and $E$.
\end{remark}

\begin{proof}[Proof of Theorem \ref{thm:radial-projections-dim}]
The proof is similar to that of Theorem \ref{thm:main-distance-sets}, but (in dimension $d>2$) we need to appeal to Theorem \ref{thm:main-continuous-higher-rank} instead.

As usual we may assume that $A$ and $E$ are disjoint compact subsets of the unit ball of positive Hausdorff measure in their dimensions.  We consider the parametrized family of smooth maps $\{e_y(x): y\in E\}$ defined on $A$. The range of these maps is $S^{d-1}$, but we can identify it with $\R^{d-1}$ in local coordinates. For example, by restricting and rotating $A$ and $E$, we may assume that $y_d-x_d\neq 0$ for all $x\in A,y\in E$, and work with the map $\wt{e}_y(x)=(y_i-x_i/(y_d-x_d))_{i=1}^{d-1}$ instead.

Our aim is to apply Theorem \ref{thm:main-continuous-higher-rank} (in the case $d=2$, we can apply Theorem \ref{thm:main-continuous}). It is easy to see that the map $V_x(y)$ featuring in Theorem \ref{thm:main-continuous-higher-rank} is given by
\[
V_x(y) = e_x(y)^{\perp}.
\]
Recalling \eqref{eq:dist-hyperplane-line}, this implies that, for $\theta\in S^{d-1}$,
\begin{equation} \label{eq:explicit-expression-distance}
d(V_x(y),\langle \theta\rangle ) = |\langle e_x(y), \theta\rangle | < r \Longleftrightarrow e_x(y)\in (\theta^{\perp})^{(r)}.
\end{equation}

Let $\mu,\nu$ be Frostman measures on $A,E$ of exponents $\alpha,\kappa$. Then
\begin{equation} \label{eq:Frostman-mu-nu}
\begin{split}
\mu(V^{(r)}) &\lesssim r^{\alpha-(d-2)}\\
\nu(V^{(r)}) &\lesssim r^{\kappa-(d-2)}
\end{split}
\end{equation}
for all $V\in \mathbb{A}(d,d-2)$ and $r\in (0,1]$. We may assume that $\mu$  assigns zero mass to all affine hyperplanes, for otherwise there is a hyperplane $P$ such that $\hdim(A\cap P)=\hdim(A)$ and, by assumption, there is a point $y\in E\setminus P$. It is then clear that
\[
\hdim(e_y(A\cap P))=\hdim(A\cap P)=\hdim(A),
\]
and we are done.

Suppose first that $\nu$ assigns positive mass to a hyperplane $P$. Then we can assume that $E\subset P$, and because $\mu(P)=0$, we may also assume that $A$ is disjoint from $P$. Using this and \eqref{eq:explicit-expression-distance}, it follows that, for $x\in A$,
\[
V_x\nu(\cV(\langle\theta\rangle,r))=\nu\left\{ e_x^{-1}\left((\theta^{\perp})^{(r)}\right)\right\}\le \nu\left( V_\theta^{(C_x r)}\right),
\]
where $V_\theta=e_x^{-1}(\theta^{\perp})\cap P\in \mathbb{A}(d,d-2)$, and the constant $C_x$ depends on the distance from $x$ to $P$. According to \eqref{eq:Frostman-mu-nu}, the hypotheses of Theorem \ref{thm:main-continuous-higher-rank} are met, and we get that there is $y\in E$ such that
\[
\hdim(e_y A) \ge \frac{(d-1)\hdim(A)}{d}+\eta,
\]
where $\eta=\eta(\kappa-(d-2),\alpha)$ is the number given by Theorem \ref{thm:main-continuous-higher-rank}.

Suppose now that $\nu$ gives zero mass to every hyperplane. We can then apply Theorem \ref{thm:radial-positive-dim} with $k=d-1$ to obtain a parameter $\kappa'=\kappa'_d(\kappa,\alpha)>0$ and a set $L$ with $\mu(L)\ge 1/2$, such that for each $x\in L$ there is a set $K(x)\subset A$ with $\nu(K(x))\ge 1/2$ so that
\begin{equation} \label{eq:hyperplane-decay}
e_x\nu_{K(x)}(H^{(r)}) \le O_{\nu,\mu}(1) \, r^{\kappa'}\quad\text{ for all }H\in \mathbb{G}(d,d-1), 0< r\le 1.
\end{equation}
Using \eqref{eq:explicit-expression-distance}, we deduce that
\[
V_x \nu_{K(x)}(\cV(\langle\theta\rangle,r)) = e_x\nu_{K(x)} ((\theta^{\perp})^{(r)}).
\]
Recalling \eqref{eq:hyperplane-decay}, this shows that the hypotheses of Theorem \ref{thm:main-continuous-higher-rank} are verified, with $\kappa'$ in place of $\kappa$. Hence, if $\eta=\eta_d(\kappa',\alpha)>0$ is the value provided by Theorem \ref{thm:main-continuous-higher-rank}, then there is $y\in\supp(\nu)\subset E$ such that
\[
\hdim(e_y(A)) \ge \frac{(d-1)\alpha}{d}+\eta.
\]
This is what we wanted to prove.
\end{proof}

\appendix
\section{Entropy of projections}

\label{sec:entropy-of-projections}

\subsection{Entropy basics}

In this section we prove Proposition \ref{prop:good-bound-from-below}. We start by reviewing some basic facts about (Shannon) entropy.   Recall that if $\nu\in\cP(\R^d$) and $\mathcal{A}$ is a finite partition of $\R^d$, up to a $\nu$-null set, then the entropy of $\nu$ with respect to $\mathcal{A}$ is given by
\[
H(\nu,\mathcal{A}) = -\sum_{A\in\mathcal{A}} \nu(A) \log(\nu(A)),
\]
with the usual convention $0\cdot \log 0=0$, and in the case $\mathcal{A}=\cD_m$, we write $H_m(\nu) = H(\nu,\cD_m)$.  Likewise, we define the conditional entropy with respect to the finite measurable partition $\mathcal{G}$ by
\[
H(\mu,\cA|\mathcal{G}) = \sum_{G\in\mathcal{G}:\mu(G)>0} \mu(G) H(\mu_G,\cA).
\]
It follows from the concavity of the logarithm that one always has
\begin{equation} \label{eq:entropy-gral-bounds}
H(\nu,\mathcal{A}|\mathcal{G}) \le \log|\mathcal{A}|.
\end{equation}

The following are some further elementary properties of entropy that we will need to call upon:
\begin{enumerate}[(A)]
\item \label{enum:equivalent-partitions}  If $\cF, \cG$ have the property that each element of $\cF$ hits at most $N$ elements of $\cG$ and vice-versa, then
\[
|H(\mu,\cF)-H(\mu,\cG)| \le \log N.
\]
\item \label{enum:refining-partitions} If $\cG$ refines $\cF$ (that is, each element of $\cF$ is a union of elements in $\cG$), then
\[
H(\mu,\cF|\cG) = H(\mu,\cG)-H(\mu,\cF).
\]
\item \label{enum:concavity-of-entropy} Conditional entropy is concave as a function of the measure: for $t\in [0,1]$,
\[
H(t\mu+(1-t)\nu,\cF|\cG) \ge t H(\mu,\cF|\cG)+(1-t)H(\nu,\cF|\cG).
\]
\end{enumerate}

\subsection{Multiscale estimates for the entropy of projections}

The following proposition provides a lower bound for the entropy of a smooth image in terms of entropies of linear images of conditional measures at certain scales. Recall that we denote orthogonal projection onto $V\in \mathbb{G}(d,k)$ by $P_V$. Furthermore, if $A:\R^d\to\R^k$ is a linear map of rank $k$, we denote by $J_k(A)$ the absolute value of the determinant of $A|_{\ker(A)^{\perp}}: \ker(A)^{\perp}\to \R^k$.

\begin{prop} \label{prop:entropy-of-image-measure}
Fix $1\le k<d$. Let $\mu\in\cP([0,1)^d)$ and let $[A_i,B_i)_{i=1}^q$ be disjoint sub-intervals of $(0,m]$ such that $B_i\le 2 A_i$. Let $F:U\to \R^k$ be a $C^2$ map defined in a neighborhood of $\supp(\mu)$, such that $DF(x)$ has full rank $k$ for all $x\in\supp(\mu)$. Denote
\[
V(x) = \ker(DF(x))^{\perp}.
\]
Then
\begin{equation} \label{eq:lower-bound-entropy}
 H(F\mu,\cD_m) \ge - O_{F,d,k}(q)  + \sum_{i=1}^{q} \sum_{Q\in\cD_{A_i}} \mu(Q)  H\left(P_{V(x_Q)}\mu^Q,\cD_{B_i-A_i}\right),
\end{equation}
where  $x_Q$ is an arbitrary point in $Q$. The constant implicit in $O_{F,d,k}(q)$ depends only on $d$, $k$, $\|F\|_{C^2}$ and $\inf_{x\in\supp(\mu)} J_k(F'(x))$; in particular, it can be taken uniform in a $C^2$-neighborhood of $F$.
\end{prop}

The proof of the proposition depends on a linearization argument, which helps explain the assumption that $F$ is $C^2$ and has no singular points. The hypothesis $B_j\le 2A_j$ comes from linearization, and can be dropped if $F$ is already linear.

\begin{lemma} \label{lem:linear-to-nonlinear-entropy}
Under the assumptions of Proposition \ref{prop:entropy-of-image-measure}, if $Q\in\cD_{A_j}$ has positive $\mu$-measure and $x\in Q$, then
\begin{equation} \label{eq:linear-to-nonlinear-entropy}
\left| H\big(F(\mu_Q) ,\cD_{B_j}\big) - H\big(P_{V(x)}(\mu_Q),\cD_{B_j}\big) \right|\lesssim_{F,d} 1 ,
\end{equation}
\end{lemma}
\begin{proof}
All implicit constants are allowed to depend on $d$. The claim can be rewritten as
\[
\left| H(\mu_Q ,F^{-1}(\cD_{B_j})) - H(\mu_Q,P_{V(x)}^{-1}(\cD_{B_j})) \right|\lesssim_F 1.
\]
Let $L_{x}(z)= F(x)+DF(x)\cdot(z-x)$ be the affine approximation to $F$ centred at $x$. Note that
\[
DF(x)|_{V(x)} P_{V(x)} = DF(x),
\]
and $DF(x)|_{V(x)}$ is an isomorphism from $V(x)$ to $\R^k$. By the assumption that $F$ has no singular points, the determinant of this isomorphism (given by $J_k(DF(x))$ in absolute value) is bounded away from $0$ and $\infty$ for $x\in \supp(\mu)$. Hence it follows from Property \eqref{enum:equivalent-partitions} that
\[
\left| H(\mu_Q ,L_x^{-1}(\cD_{B_j})) - H(\mu_Q,P_{V(x)}^{-1}(\cD_{B_j})) \right|\lesssim_F 1.
\]
So it is enough to show that
\begin{equation} \label{eq:equivalent-partitions}
\left| H(\mu_Q ,F^{-1}(\cD_{B_j})) - H(\mu_Q,L_x^{-1}(\cD_{B_j})) \right|\lesssim_F 1.
\end{equation}
This is just a consequence of Taylor's formula. Indeed, since $F$ is $C^2$,
\[
|F(z)-L_x(z)| \le  O_{\|F\|_{C^2}}(|z-x|^2) \le O_{\|F\|_{C^2}}(2^{-B_j})\quad\text{for }z\in Q,
\]
using the assumption $B_j\le 2 A_j$. This implies that each element of $F^{-1}(\cD_{B_j})$ hitting $Q$ intersects $\lesssim_{F} 1$ elements of $L_x^{-1}(\cD_{B_j})$, and vice-versa. Property \eqref{enum:equivalent-partitions} above yields that \eqref{eq:equivalent-partitions} is verified, and this establishes the claim \eqref{eq:linear-to-nonlinear-entropy}.
\end{proof}

We can now conclude the proof of Proposition \ref{prop:entropy-of-image-measure}.
\begin{proof}[Proof of Proposition \ref{prop:entropy-of-image-measure}]
\begin{align*}
H(F\mu,\cD_m) &\overset{\eqref{enum:refining-partitions} }{\ge} -H(F\mu,\cD_0) + \sum_{i=1}^{q}  H(F\mu,\cD_{B_i}|\cD_{A_i}) \\
&\overset{\eqref{eq:entropy-gral-bounds}}{\ge} -O_{F,d}(1)+\sum_{i=1}^{q} H\left(\sum_{Q\in\cD_{A_i}} \mu(Q) F(\mu_Q) ,\cD_{B_i}|\cD_{A_i}\right)\\
&\overset{\eqref{enum:concavity-of-entropy}}{\ge} -O_{F,d}(1) + \sum_{i=1}^{q} \sum_{Q\in\cD_{A_i}} \mu(Q) H(F(\mu_Q) ,\cD_{B_i}|\cD_{A_i})\\
&\overset{\eqref{enum:refining-partitions}, \eqref{eq:entropy-gral-bounds} }{\ge}- O_{F,d}(q)+ \sum_{i=1}^{q} \sum_{Q\in\cD_{A_i}} \mu(Q) H(F(\mu_Q),\cD_{B_i})\\
&\overset{\eqref{eq:linear-to-nonlinear-entropy}}{\ge} -O_{F,d}(q) + \sum_{i=1}^{q} \sum_{Q\in\cD_{A_i}} \mu(Q)\left( H(P_{V(x_Q)}\mu_Q,\cD_{B_i}) - O_F(1)\right)\\
&\overset{P_{V(x_Q)} \text{ linear}}{=} -O_{F,d}(q)+\sum_{i=1}^{q} \sum_{Q\in\cD_{A_i}} \mu(Q) H(\Pi_{V(x_Q)}\mu^Q ,\cD_{B_i-A_i}).
\end{align*}
An inspection of the proof shows that the constants depending on $F$ do so in the way described in the statement.
\end{proof}

We can now complete the proof of Proposition \ref{prop:good-bound-from-below}, which can be seen as a ``robust'' version of Proposition \ref{prop:entropy-of-image-measure}; we repeat the statement for the reader's convenience.
\begin{prop}
Under the same assumptions and notation of Proposition \ref{prop:entropy-of-image-measure}, if $\nu\in\cP(\R^d)$ satisfies $\nu\le \Delta\mu$, then
\[
 H_m(F\nu) \ge - O_{F,d}(q)   + \sum_{i=0}^{q-1} \sum_{Q\in\cD_{A_i}} \nu(Q)  H_{B_i-A_i}^{m \Delta}\left(P_{V(x_Q)}\mu^Q\right).
\]
\end{prop}
\begin{proof}
Fix $i\in \{0,\ldots,q-1\}$, and note that
\begin{equation} \label{eq:low-density-has-small-measure}
\sum\{  \nu(Q): Q\in\cD_{A_i}, \nu(Q) < \tfrac{1}{m} \mu(Q)\} < \tfrac{1}{m}.
\end{equation}
Suppose $\nu(Q) \ge \tfrac{1}{m}\mu(Q)>0$ for a given $Q\in\cD_{A_i}$. Then
\[
\nu_Q(S) = \frac{\nu(Q\cap S)}{\nu(Q)}  \le \frac{\Delta \mu(Q\cap S)}{\tfrac{1}{m}\mu(Q)} = m \Delta \mu_Q(S)
\]
for any Borel set $S\subset [0,1)^2$. We deduce that $\Pi\nu_Q\le m\Delta \Pi\mu_Q$ for any linear map $\Pi$, and hence
\[
H_{B_i-A_i}\left(P_{V(x_Q)}\nu^Q\right) \ge  H^{m \Delta}_{B_i-A_i}\left(P_{V(x_Q)}\mu^Q\right),
\]
always assuming that $\nu(Q)\ge \tfrac{1}{m}\mu(Q)>0$ and $Q\in \cD_{A_i}$.

On the other hand, for fixed $i$, from \eqref{eq:low-density-has-small-measure} and the trivial bound $H_p(\cdot) \le k(p+O(1))$ for measures supported on a ball of radius $O(1)$ in $\R^k$, we get
\[
\sum_{Q\in\cD_{A_i}:\nu(Q)<\tfrac{1}{m} \mu(Q)} \nu(Q)  H_{B_i-A_i}^{m \Delta}\left(P_{V(x_Q)}\mu^Q\right) \le k(B_i-A_i+O_d(1))/m.
\]
Splitting (for each $i$) the sum $\sum_{Q\in\cD_{A_i}}$ in Proposition \ref{prop:entropy-of-image-measure} into the cubes with $\nu(Q) \ge \tfrac{1}{m}\mu(Q)$ and $\nu(Q)< \tfrac{1}{m}\mu(Q)$, we get the desired result.
\end{proof}

\section{Spherical projections}
\label{sec:spherical-projections}

In this section we prove Theorem \ref{thm:radial-positive-dim}. In fact, we will prove the following more quantitative statement:
\begin{theorem} \label{thm:radial-positive-dim-app}

Fix $1\le k<d$. Given $\kappa,\alpha>0$ there are $c,\eta>0$ (depending continuously on $\alpha,\kappa$) such that the following holds.

Let $\mu,\nu\in\cP(B^d(0,1))$ be measures satisfying decay conditions
\begin{align*}
\mu(V^{(r)}) &\le C_\mu r^{\kappa},\\
\nu(V^{(r)}) &\le C_\nu r^{\alpha},
\end{align*}
for all $V\in \mathbb{A}(d,k-1)$ and $0<r\le 1$. (If $k=1$, $V^{(r)}$ is a ball of radius $r$.) Fix $\e\in(0,1)$ and suppose there is $\tilde{\delta}>0$ such that
\begin{equation} \label{eq:small-plate-base}
\nu(V^{(\tilde{\delta})}) \le c\e, \quad V\in\mathbb{A}(d,k).
\end{equation}
Then  for all $x$ in a set $E$ of $\mu$-measure $\ge 1-\e$ there is a set $K(x)$ with $\nu(K(x))\ge 1-\e$ such that
\begin{equation} \label{eq:spherical-proj-non-concentration-app}
e_x(\nu_{K(x)})(W^{(r)}) \le r^\eta \quad r\in (0,r_0], W\in \mathbb{G}(d,k),
\end{equation}
where $r_0>0$ depends only on $d, C_\mu, C_\nu, \alpha, \tilde{\delta},\e$.

Finally, the set $\{ (x,y) : x\in E, y\in K(x) \}$ is compact.
\end{theorem}

Theorem \ref{thm:radial-positive-dim} follows taking $\e=1/2$ and observing that, by compactness, if $\nu$ gives zero mass to all affine $k$-planes, then there is $\delta_0$ such that \eqref{eq:small-plate-base} holds. We follow closely the ideas from \cite[Section 2]{Orponen19}, although there are some differences: we work in arbitrary dimension, while Orponen works only in the plane. Also, Orponen obtains a single $x$ such that \eqref{eq:spherical-proj-non-concentration-app} holds.  On the other hand, some aspects of the proof are simplified; in particular, we avoid having to argue by induction in the scale, which reduces the number of parameters to keep track of in the proof (I thank Hong Wang for suggesting this simplification and allowing me to include it in this article).

The proof of the theorem depends on a single-scale variant that we state and prove first. By a $k$-plate in $\R^d$ we mean a set of the form
\[
W^{(\delta)} \cap B^d(0,1)
\]
with $W\in \mathbb{A}(d,k)$. In particular, $0$-plates are balls and $1$-plates are tubes (intersected with $B(0,1)$).  We refer to $\delta$ as the width of the plate, and denote the family of all $k$-plates of width $\delta$ intersecting a set $E$ by $\mathcal{T}_k(E,\delta)$.

\begin{prop} \label{prop:radial-single-scale}
Given $\kappa,\alpha>0$ and $1\le k\le d-1$, the following holds if $\eta\le\eta_0(\kappa,\alpha)$, and $\delta$ is small enough depending on $C_\mu, C_\nu$ appearing below and all the other parameters.

Let $\nu,\mu\in\cP(B^d(0,1))$ be measures satisfying decay conditions
\begin{align*}
\mu(V^{(r)}) &\le C_\mu r^{\alpha} & (V\in \mathbb{A}(d,k-1), 0<r\le 1), \\
\nu(V^{(\delta)}) &\le C_\nu\delta^{\kappa} & (V\in \mathbb{A}(d,k-1)).
\end{align*}

Then there exist:
\begin{itemize}
\item A set $E\subset\supp(\mu)$ with $\mu(\R^d\setminus E)\le \delta^{\kappa\alpha/2}$,
\item For each $x\in E$, a set $P(x)$, which is either empty or a $k$-plate in $\cT_k(x,\delta^{\eta/2})$,
\end{itemize}
such that $\{ (x,y): x\in E, y\in \supp\nu\setminus P(x)\}$ is compact, and
\[
\nu(W)\le \delta^{\eta} \quad\text{for all } W\in\cT_k(x,\delta)\cap \cT_k(\supp(\nu)\setminus P(x),\delta).
\]
\end{prop}
In turn, the proposition relies on the following lemma stating that all $\delta$-plates of too large measure are contained in a relatively small number of relatively thin plates.
\begin{lemma} \label{lem:structure-bad}
Let $\nu\in\cP(B^d(0,1))$ satisfy $\nu(W)\le  C_\nu\delta^{\kappa}$ for all $(k-1)$-plates $W$ of width $\delta$. Fix $\eta>0$.

Then there exists a family of $\lesssim_d C_\nu \delta^{-\eta}$ $k$-plates $T_j$ of width $\delta^{\kappa-2\eta}$, such that any $k$-plate $W$ of width $\delta$ with $\nu(W)\ge \delta^{\eta}$ is contained in one of the plates $T_j$.
\end{lemma}
\begin{proof}
We construct a sequence $Y_1,\ldots, Y_M$ of $k$-plates of width $\delta$ as follows. To begin, pick (if it exists) a $k$-plate $Y_1$ of width $\delta$ such that $\nu(Y_1)\ge\delta^\eta$. Now suppose $Y_1,\ldots,Y_m$ are $k$-plates of width $\delta$ such that:
\begin{enumerate}
\item $\nu(Y_i) \ge \delta^\eta$,
\item $\nu(Y_i\cap Y_j) \le \delta^{2\eta}/2$ for $1\le i<j\le m$.
\end{enumerate}
If there exists $Y_{m+1}$ such that the collection $(Y_1,\ldots,Y_{m+1})$ still satisfies properties (1)-(2), we add it to the list; otherwise we stop. A standard $L^2$ argument implies that we have to stop after $\le 2\delta^{-\eta}$ steps. Indeed, if $f=\sum_{i=1}^m \mathbf{1}_{Y_i}$, then
\begin{align*}
(\sum_i \nu(Y_i))^2 &=  \left(\int f\,d\nu\right)^2 \le  \int f^2\,d\nu \\
&=   \sum_{1\le i,j\le m} \nu(Y_i\cap Y_j) < \left(\sum_{i}\nu(Y_i)\right) + m^2 \delta^{2\eta}/2.
\end{align*}
Let $S=\sum_{i=1}^m \nu(Y_i)\ge m\delta^\eta$. We have seen that $S^2-S< m^2\delta^{2\eta}/2$. If $m > 2\delta^{-\eta}$, then $S^2-S \ge S^2/2$, and we deduce that
\[
m^2 \frac{1}{2}\delta^{2\eta} > \frac{1}{2}(\sum_i \nu(Y_i))^2 \ge \frac{1}{2} m^2 \delta^{2\eta},
\]
which is a contradiction. Let, then, $(Y_i)_{i=1}^M$ be the final family so obtained, with $M\le 2\delta^{-\eta}$.

Now fix a $k$-plate $W$ of width $\delta$ such that $\nu(W)\ge \delta^{\eta}$. Hence
\[
\nu(W\cap Y_j) \ge \delta^{2\eta}/2 \quad\text{ for some } Y_j,
\]
for otherwise we could add $W$ to the list of $Y_i$. On the other hand, $W\cap Y_j$ is contained in a box of size at most a constant $C_d$ times
\[
\underbrace{1\times \cdots \times 1}_{k-1 \text{ times}} \times \delta /\angle(W,Y_j) \times \underbrace{\delta\times \cdots \times\delta}_{d-k \text{ times}},
\]
where $\angle(\cdot,\cdot)$ denotes the largest principal angle between the $k$-planes determining the corresponding plates. Hence the assumption on $\nu$ yields
\[
\nu(W\cap Y_j) \lesssim_d \frac{C_\nu\delta^{\kappa}}{\angle(W,Y_j)}.
\]
Comparing the upper and lower bounds on $\nu(W\cap Y_j)$, we deduce that $\angle(W,Y_j)\lesssim_d  C_\nu\delta^{\kappa-2\eta}$ and, therefore, $W$ is contained in the plate $T_j$ centered at (the plane defining) $Y_j$ but of width $\lesssim_d C_\nu \delta^{\kappa-2\eta}$. This is what we wanted to prove.
\end{proof}

\begin{proof}[Proof of Proposition \ref{prop:radial-single-scale}]
Let
\[
\bad = \{ x\in B^d(0,1): \nu(W) \ge \delta^{\eta} \text{ for some } W\in\cT_k(x,\delta)\}.
\]
Let $T_j$ be the plates provided by Lemma \ref{lem:structure-bad}, and define
\[
\bad\bad = \{ x\in \bad: x\in T_i\cap T_j \text{ for some }i,j \text{ such that } \angle(T_i,T_j) \ge \delta^\eta \}.
\]
Let $r= \delta^{k-2\eta}$. As before, $T_i\cap T_j$ is contained in the union of $\lesssim_d 1$ boxes of size
\[
\underbrace{1\times \cdots \times 1}_{k-1 \text{ times}} \times r /\angle(T_i,T_j) \times \underbrace{r\times \cdots \times r}_{d-k \text{ times}},
\]
and hence, now using the non-concentration assumption on $\mu$,
\begin{align*}
\mu(\bad\bad) &\le \sum\{\mu(T_i\cap T_j):  \angle(T_i,T_j) \ge \delta^\eta\}\\
&\lesssim_d C_\mu C_\nu\delta^{-3\eta} \delta^{(\kappa-2\eta)\alpha} = C_\mu C_\nu\delta^{\kappa\alpha-\eta(3+2\kappa)}.
\end{align*}
Hence, if $\eta$ is small enough in terms of $\kappa,\alpha$, and $\delta$ is small enough in terms of $C_\mu, C_\nu$ and $d$, then
\[
\mu(\bad\bad) \le \delta^{\kappa\alpha/2}/2.
\]
Let $E=\supp(\mu)\setminus\bad\bad$, and fix $x\in E$. If $x\notin\bad$, we  can just take $P(x)=\varnothing$, so assume that $x\in\bad$. By Lemma \ref{lem:structure-bad},  $\bad$ is covered  by the $2 \delta^{-\eta}$ plates $T_j\in\cT_k(E,O(C_\nu)\delta^{\kappa-2\eta})$. Let $E_j=E\cap T_j\setminus  \cup_{i=1}^{j-1} T_i$, and for $x\in E_j$ let $P(x)$ be the $\delta^{\eta/2}$-plate centered at $T_j$. By passing to a subset of $E$, we can ensure that $E\cap\bad$ and the $E_j$ are compact, and still $\mu(\R^d\setminus E)\le \delta^{\kappa\alpha/2}$. This ensures that $\{ (x,y): x\in E, y\notin P(x)\}$ is compact.

Now fix $x\in E_j$ and suppose $W\in\cT_k(x,\delta)\cap\cT_k(\supp(\nu)\setminus P(x),\delta)$. Then, provided $\eta$ is taken small enough that $\kappa-2\eta\ge\eta$ and $\delta$ is small, the plate $W$ cannot be contained in any of the plates $T_i$ making an angle smaller than $\delta^\eta$ with $T_j$. Since $x\in T_j\setminus\bad\bad$, we conclude that the plate $W$ cannot be contained in any of the $T_i$. Lemma \ref{lem:structure-bad} now implies that $\nu(W)\le \delta^\eta$, completing the proof.
\end{proof}

We can now complete the proof of Theorem \ref{thm:radial-positive-dim-app}.
\begin{proof}[Proof of Theorem \ref{thm:radial-positive-dim-app}]
Let $\eta>0$ be the number given by Proposition \ref{prop:radial-single-scale}. Making $\eta$ slightly smaller if needed, we may assume that $\eta/2=2^{-N}$ for some integer $N$. We will prove the claim with $\eta/2$ in place of $\eta$.  We take $c=1/(2N)$, so by assumption there is $\delta_0>0$ such that $\nu(W)\le \e/(2N)$ for all $k$-plates $W$ of width $\delta_0^{\eta/2}$. Making $\delta_0$ smaller if needed, we can also ensure that
\begin{equation} \label{eq:sum-mass-tubes-small}
\sum_{n=1}^\infty \delta_0^{2^n\cdot \eta}  < \e/2.
\end{equation}
Finally, we also take $\delta_0$ small enough that Proposition \ref{prop:radial-single-scale} applies to all $\delta\le \delta_0$.

Write $\delta_n=\delta_0^{2^n}$. For each $n$, let $E_n$ and $P_n(x), x\in E_n$ be the sets provided by Proposition \ref{prop:radial-single-scale} for the scales $\delta_n$. Let
\begin{align*}
E &= \cap_{n=0}^\infty E_n,\\
K(x) &= \supp(\nu)\setminus \cup_{n=0}^\infty P_n(x).
\end{align*}
Then $\{ (x,y):x\in E, y\in K(x)\}$ is compact as an intersection of compact sets, and $\nu(E)\ge 1-\e$ by \eqref{eq:sum-mass-tubes-small}, taking $\eta<\alpha\kappa/2$. Fix $x\in E$ for the rest of the proof.

By Proposition \ref{prop:radial-single-scale}, and choosing for each $r<\delta_1$ the $n$ such that $\delta_n^2=\delta_{n-1}< r\le\delta_n$, we get that $\nu(K(x)\cap W)\le r^{\eta/2}$. This yields \eqref{eq:spherical-proj-non-concentration-app}.

It remains to show that $\nu(K(x))\ge 1-\e$. By construction, $\nu(P_n(x))\le \e/(2N)$ for all $n\le N$. Recall that $\eta/2=2^{-N}$, and hence $\delta_n^{\eta/2}=\delta_{n-N}\le \delta_1$ for all $n>N$. Applying Proposition \ref{prop:radial-single-scale} at scale $\delta_{n-N}$ with $P_n(x)$ in place of $W$,  we get
\[
\nu(P_n(x)\setminus \cup_{j=1}^{n-1} P_j(x)) \le \nu(P_n(x)\setminus P_{n-N}(x)) \le \delta_{n-N}^\eta\quad\text{for all }n>N.
\]
Adding up and recalling \eqref{eq:sum-mass-tubes-small}, we see that indeed $\nu(K(x))\ge 1-\e$, and this completes the proof.
\end{proof}


\end{document}